\documentclass{amsart}

\usepackage{bbm}
\usepackage{amsmath,amsthm,amstext,amssymb}
\usepackage{hyperref}
\usepackage{enumerate,enumitem}

\usepackage
 [a4paper, margin=1.2in]{geometry}

\theoremstyle{plain}
\newtheorem{theorem}{Theorem}[section]
\newtheorem{lemma}[theorem]{Lemma}

\newtheorem{corollary}[theorem]{Corollary}
\newtheorem{proposition}[theorem]{Proposition}

\newtheorem*{claim*}{Claim}
\theoremstyle{definition}
\newtheorem{definition}[theorem]{Definition}
\newtheorem{remark}[theorem]{Remark}

\newtheorem{question}[theorem]{Question}

\newcommand{\betrag}[1]{\vert{#1}\vert}

\newcommand{\dom}[1]{{{\rm{dom}}(#1)}}
\newcommand{\crit}[1]{{{\rm{crit}}\left({#1}\right)}}
\newcommand{\cof}[1]{{{\rm{cof}}(#1)}}

\newcommand{\ran}[1]{{{\rm{ran}}(#1)}}

\newcommand{\tc}[1]{{\rm{tc}}({#1})}
\newcommand{\length}[1]{{\rm{lh}}({#1})}
\newcommand{\POT}[1]{{\mathcal{P}}({#1})}

\newcommand{\map}[3]{{#1}:{#2}\longrightarrow{#3}}
\newcommand{\Map}[5]{{#1}:{#2}\longrightarrow{#3};~{#4}\longmapsto{#5}}

\newcommand{\Set}[2]{\{{#1}~\vert~{#2}\}}
\newcommand{\seq}[2]{\langle{#1}~\vert~{#2}\rangle}

\newcommand{\goedel}[2]{{\prec}{#1},{#2}{\succ}}

\newcommand{\anf}[1]{{\text{``}\hspace{0.3ex}{#1}\hspace{0.3ex}\text{''}}}

\newcommand{\HH}[1]{{\rm{H}}(#1)}
\newcommand{\Ult}[2]{{\mathrm{Ult}}({#1},{#2})}

\newcommand{\Add}[2]{{\rm{Add}}({#1},{#2})}
\newcommand{\Col}[2]{{\rm{Col}}({#1},{#2})}

\newcommand{\id}{{\rm{id}}}
\newcommand{\Lim}{{\rm{Lim}}}
\newcommand{\On}{{\rm{Ord}}}

\newcommand{\LL}{{\rm{L}}}
\newcommand{\HOD}{{\rm{HOD}}}

\newcommand{\ZF}{{\rm{ZF}}}
\newcommand{\ZFC}{{\rm{ZFC}}}

\newcommand{\CH}{{\rm{CH}}}
\newcommand{\AC}{{\rm{AC}}}
\newcommand{\DC}{{\rm{DC}}}
\newcommand{\AD}{{\rm{AD}}}

\newcommand{\PFA}{{\sf{PFA}}}
\newcommand{\MM}{{\sf{MM}}}
\newcommand{\BMM}{{\sf{BMM}}}
\newcommand{\BPFA}{{\sf{BPFA}}}

\newcommand{\CCC}{{\mathbb{C}}}

\newcommand{\PPP}{{\mathbb{P}}}

\newcommand{\RRR}{{\mathbb{R}}}

\newcommand{\TTT}{{\mathbb{T}}}

\newcommand{\KK}{{\rm{K}}}

\newcommand{\VV}{{\rm{V}}}

\newcommand{\calA}{\mathcal{A}}
\newcommand{\calB}{\mathcal{B}}

\newcommand{\calL}{\mathcal{L}}

\title{Partition properties for simply definable colourings}

\author{Philipp L\"ucke}
\address{Mathematisches Institut\\Rheinische Friedrich-Wilhelms-Universit\"at Bonn\\En\-de\-nicher Allee 60\\53115 Bonn\\Germany}
\email{pluecke@math.uni-bonn.de}


\begin{document}

\begin{abstract}
 We study partition properties for uncountable regular cardinals that arise by restricting partition properties defining large cardinal notions to classes of simply definable colourings.  
 We show that both large cardinal assumptions and forcing axioms imply that there is a homogeneous closed unbounded subset of $\omega_1$ for every colouring of the finite sets of countable ordinals that is definable by a $\Sigma_1$-formula that only uses the cardinal $\omega_1$ and real numbers as parameters. 
Moreover, it is shown that certain large cardinal properties cause analogous partition properties to hold at the given large cardinal and these implications yield natural examples of inaccessible cardinals that possess strong partition properties for $\Sigma_1$-definable colourings and are not weakly compact. 
In contrast, we show that $\Sigma_1$-definability behaves fundamentally different at $\omega_2$ by showing that various large cardinal assumptions and \emph{Martin's Maximum} are compatible with the existence of a colouring of pairs of elements of $\omega_2$ that is definable by a $\Sigma_1$-formula with parameter $\omega_2$ and has no uncountable homogeneous set. 
Our results will also allow us to derive tight bounds for the consistency strengths of various partition properties for definable colourings. 
Finally, we use the developed theory to study the question whether certain homeomorphisms that witness failures of weak compactness at small cardinals can be simply definable.  
\end{abstract}

\maketitle



\section{Introduction}

Many important results in contemporary set theory show that canonical extensions of the axioms of \emph{Zermelo--Fraenkel set theory} $\ZFC$ by large cardinal assumptions or forcing axioms cause small uncountable cardinals to satisfy strong fragments of large cardinal properties. 
For example, a classical result of Baumgartner shows that the \emph{Proper Forcing Axiom $\PFA$} implies the non-existence of $\aleph_2$-Aronszajn trees  (see {\cite[Theorem 7.2]{MR776640}}) and hence this axiom causes the second uncountable cardinal $\omega_2$ to possess a strong fragment of weak compactness.
Another example of such a results is given by a theorem of Woodin that shows that the existence of a Woodin cardinal $\delta$ causes the first uncountable cardinal $\omega_1$ to possess non-trivial fragments of almost hugeness by showing that, in a generic extension $\VV[G]$ of the ground model $\VV$, there is an elementary embedding $\map{j}{\VV}{M}$ with critical point $\omega_1^\VV$ that satisfies $j(\omega_1^\VV)=\delta$ and $({}^{{<}\delta}M)^{\VV[G]}\subseteq M$ (see {\cite[Theorem 2.5.8]{MR2069032}}).

In this paper, we focus on large cardinal properties defined with the help of partition properties and fragments of these properties that arise through restrictions of the considered colourings. 
 Remember that, if $X$ is a set and $k<\omega$, then we let $[X]^k$ denote the collection of all $k$-element subsets of $X$ and, given a function $c$ with domain $[X]^k$, a subset $H$ of $X$ is \emph{$c$-homogeneous} if $c\restriction{[H]^k}$ is constant. 
 A classical result of Erd\H{o}s and Tarski then shows that an uncountable cardinal $\kappa$ is weakly compact if and only if for every function $\map{c}{[\kappa]^2}{2}$, there is a $c$-homogeneous subset of $\kappa$ of cardinality $\kappa$. 
 %
 %
 The other large cardinal property defined through partition properties that is relevant for this paper is the concept of \emph{Ramseyness} introduced by Erd\H{o}s and Hajnal in \cite{MR0141603}. 
 They defined an infinite cardinal $\kappa$ to be \emph{Ramsey} if for every function $\map{c}{[\kappa]^{{<}\omega}}{\gamma}$ that sends elements of the collection $[\kappa]^{{<}\omega}$ of all finite subsets of $\kappa$ to elements of an ordinal $\gamma<\kappa$, there is a subset $H$ of $\kappa$ of cardinality $\kappa$ that is $(c\restriction[\kappa]^k)$-homogeneous for all $k<\omega$.

The work presented in this paper studies the fragments of the above partition properties that are obtained by restricting these properties to definable colourings. 
 Similar restrictions have already been studied in {\cite{MR2310340}}, \cite{MR2267146} and \cite{MR1791372}, where large cardinal properties are restricted to objects that are \emph{locally} definable, i.e. subsets of $\HH{\kappa}$ that are definable over the structure $\langle\HH{\kappa},\in\rangle$. 
In contrast, we will focus on partitions that are \emph{globally} definable, i.e. subsets of $\HH{\kappa}$ that are definable over $\langle\VV,\in\rangle$. 
Our results will show that canonical extensions of $\ZFC$ by large cardinal assumptions or forcing axioms cause strong partion properties for simply definable colourings to hold at $\omega_1$ and that several well-known large cardinal provide examples of inaccessible cardinal that are not weakly compact but possess strong partition properties for simply definable colourings. 
In contrast, we show that neither large cardinal assumptions nor  forcing axioms yield similar partition properties for $\omega_2$. 
Before we formulate these properties, we make two observations that suggest that the validity of partition principles for simply definable functions can be considered intuitively plausible and also foundationally desirable.  
First, we will later show that the axioms of $\ZFC$ already prove such partition properties for colourings that are very simply definable, i.e. functions defined by formulas that only use bounded quantifiers and parameters contained in $\HH{\kappa}\cup\{\kappa\}$ (see Theorem \ref{theorem:Sigma0Partitions} below). Therefore it is reasonably to expect that canonical extensions of $\ZFC$ expand this implication to larger classes of  simply definable partitions. 
Second, if we look at the colourings that witness failures of weak compactness of small cardinals, then the constructions of these functions rely on complicated objects, like $\kappa$-Aronszajn trees or well-orderings of the collection $\HH{\kappa}$ of all sets of hereditary cardinality less than $\kappa$, that can, in general, only be obtained through applications of the \emph{Axiom of Choice $\AC$}. 
For example, the canonical colouring $\map{c}{[\omega_1]^2}{2}$ witnessing the failure of the  weak compactness of $\omega_1$ is constructed by using $\AC$ to find an injection $\map{\iota}{\omega_1}{\RRR}$ of $\omega_1$ into the real line $\RRR$ and then setting $$c(\{\alpha,\beta\})=1 ~ \Longleftrightarrow ~ \iota(\alpha) < \iota(\beta)$$ for all $\alpha<\beta<\omega_1$. 
 %
%
 %
Moreover, it is well-known that these applications of $\AC$ are actually necessary to derive failures of weak compactness at accessible cardinals, because the axioms of $\ZF$ are consistent with the statement that $\omega_1$ is weakly compact (see \cite{MR0244036}).  
 This suggests that the partitions witnessing failures of weak compactness of small cardinals should be viewed as complicated objects and therefore it seems natural to expect canonical extensions of $\ZFC$ to imply that these functions are not simply  definable. 
 %
 %

 %
 %
 %

In the following, we formulate the principles studied in this paper. 
 %
Remember that a formula in the language $\calL_\in=\{\in\}$ of set theory is a \emph{$\Sigma_0$-formula} if it is contained in the smallest collection of $\calL_\in$-formulas that contains all atomic formulas and is closed under negations, conjunctions and bounded quantification. 
 Moreover, a $\calL_\in$-formula is a \emph{$\Sigma_{n+1}$-formula} for some $n<\omega$ if it is of the form $\exists x ~ \neg\varphi$ for some $\Sigma_n$-formula $\varphi$. 
 Note that the class of all formulas that are $\ZFC$-provable equivalent to a $\Sigma_{n+1}$ is closed under existential quantification, bounded quantification, conjunctions and disjunctions. 
Finally, given sets $z_0,\ldots,z_{m-1}$, we say that a class $X$ is \emph{$\Sigma_n(z_0,\ldots,z_{m-1})$-definable} if there is a $\Sigma_n$-formula $\varphi(v_0,\ldots,v_m)$ with $X = \Set{x}{\varphi(x,z_0,\ldots,z_{m-1})}$.

\begin{definition}\label{definition:SigmanPartition}
  Given a cardinal $\kappa$, $k<\omega$ and sets $z_0,\ldots,z_{m-1}$,  a function $c$ with domain $[\kappa]^k$ is a \emph{$\Sigma_n(z_0,\ldots,z_{m-1})$-partition} if there is a $\Sigma_n$-formula $\varphi(v_0,\ldots,v_{k+m+1})$ with the property that for all $\alpha_0<\ldots<\alpha_{k-1}<\kappa$, the value $c(\{\alpha_0,\ldots,\alpha_{k-1}\})$ is the unique set $y$ such that $\varphi(\alpha_0,\ldots,\alpha_{k-1},y,\kappa,z_0,\ldots,z_{n-1})$ holds. 
\end{definition}

It is  easy to see that if $\kappa$ is a cardinal and $n>0$, then a function $c$ with domain $[\kappa]^k$ is a $\Sigma_n(z_0,\ldots,z_{m-1})$-partition if and only if the set $c$ is $\Sigma_n(\kappa,z_0,\ldots,z_{m-1})$-definable. 
 %
 Moreover, since we allow the cardinal $\kappa$ as a parameter in the definitions the graphs of our partitions, these sets will in fact be $\Delta_n$-definable, i.e. there also is a $\Pi_n$-formula (i.e. a negated $\Sigma_n$-formula) that defines the function $c$ in the above way. 
 %
 In addition, the same argument shows that, if we instead consider $\Pi_n$-definable partitions, 
then we end up with the same class of functions.

The next definition shows how we restrict weak compactness to the definable context.

\begin{definition}
 Let $\kappa$ be an uncountable regular cardinal. 
 \begin{enumerate}[leftmargin=0.9cm]
  \item Given sets $z_0,\ldots,z_{m-1}$, the cardinal $\kappa$ has the \emph{$\Sigma_n(z_0,\ldots,z_{m-1})$-colouring property} if for all $\Sigma_n(z_0,\ldots,z_{m-1})$-partitions $\map{c}{[\kappa]^2}{2}$, there is a $c$-homogeneous set of cardinality $\kappa$. 

  \item The cardinal $\kappa$ has the 
\emph{$\mathbf{\Sigma}_n$-colouring 
property\footnote{This name was chosen to avoid conflicts with the definitions of {\cite{MR2310340}} and \cite{MR2267146}, where \emph{$\mathbf{\Sigma}_n$-weakly compact cardinals} and the \emph{$\mathbf{\Sigma}_n$-partition property} were introduced. 
If $\kappa$ is an inaccessible cardinal with the $\mathbf{\Sigma}_1$-colouring property, then the fact that the set $\HH{\kappa}$ is $\Sigma_1(\kappa)$-definable implies that $\kappa$ has the $\mathbf{\Sigma}_1$-partition property (see {\cite[Definition 2.9]{MR2310340}}). 
Moreover, if $\kappa$ is an inaccessible cardinal with the $\mathbf{\Sigma}_2$-colouring property, then the set $\{\HH{\kappa}\}$ is $\Sigma_2(\kappa)$-definable and therefore $\kappa$ has the $\mathbf{\Sigma}_\omega$-partition property. 
 In addition, if $\VV=\LL$ holds and $\kappa$ is a cardinal with the $\mathbf{\Sigma}_1$-colouring property, then the set $\{\HH{\kappa}\}$ is $\Sigma_1(\kappa)$-definable, Corollary \ref{corollary:InaccessibleInLA} below shows that $\kappa$ is inaccessible and hence $\kappa$ has the $\mathbf{\Sigma}_\omega$-partition property. 
 Finally, if $\VV=\LL$ holds and $\kappa$ is a cardinal with the $\mathbf{\Sigma}_\omega$-partition property, then there is a subset $A$ of $\kappa$ with the property that the set $\{A\}$ is $\Sigma_1(\kappa)$-definable and whenever a function $\map{c}{[\kappa]^2}{2}$ is definable over $\langle\LL_\kappa,\in\rangle$ and $\lambda$ is an ordinal greater than $\kappa$ with $\LL_\lambda[A]\models\ZFC^-$, then $\LL_\lambda[A]$ contains a $c$-homogeneous subset of $\kappa$ of cardinality $\kappa$. In combination with Lemma \ref{lemma:CharacterizationsPartitionProperty} below, this shows that the assumption $\VV=\LL$ implies that every cardinal with the $\mathbf{\Sigma}_1$-colouring and the $\mathbf{\Sigma}_\omega$-partition property lies above an inaccessible cardinal with the $\mathbf{\Sigma}_\omega$-partition property. In particular, the $\mathbf{\Sigma}_\omega$-partition property does not provably imply the $\mathbf{\Sigma}_1$-colouring property.}
} 
if it has the $\Sigma_n(z_0,\ldots,z_{m-1})$-colouring property for all  $z_0,\ldots,z_{m-1}\in\HH{\kappa}$.  
 \end{enumerate}
\end{definition}

The results of this paper will show that the assumption $\VV=\HOD$ implies that all cardinals with the $\Sigma_2$-colouring  property are already weakly compact. 
Since the extension of $\ZFC$ that we consider in this paper are all compatible with the assumption $\VV=\HOD$, this result shows that the above property is most interesting for $n\leq 1$. 
 The restriction of parameters to the set $\HH{\kappa}\cup\{\kappa\}$ in the second part of the above definition is supposed to prevent partitions witnessing failures of weak compactness to be used as parameters in our definitions.  
 Note that the class of of sets that are definable by a $\Sigma_1$-formula with parameters in $\HH{\kappa}\cup\{\kappa\}$ was already studied in detail in \cite{MR3694344} and there it was shown that for certain uncountable regular cardinals $\kappa$, canonical extensions of $\ZFC$ provide a strong structure theory for this rich class of objects.

We will later show that every uncountable regular cardinal has the $\mathbf{\Sigma}_0$-colouring property and this statement cannot be strengthened to $n=1$, because cardinals with the $\mathbf{\Sigma}_1$-colouring property will turn out to be inaccessible with high Mahlo-degree in G\"odel's constructible universe $\LL$. 
But the results of this paper will allow us to show that successors of regular cardinals, successors of singular cardinals of countable cofinality and non-weakly compact inaccessible cardinals can all consistently possess the $\mathbf{\Sigma}_n$-colouring property for all $n<\omega$. 
 Moreover, we will show that many canonical extensions of $\ZFC$ cause $\omega_1$ to have the $\mathbf{\Sigma}_1$-colouring  property and $\ZFC$ alone proves that several types of non-weakly compact large cardinal have this property. 
In contrast, we will show that the influence of large cardinal assumptions and forcing axioms on $\Sigma_1$-definability at $\omega_2$ is completely different from the effect of these extensions of $\ZFC$ on $\Sigma_1$-definability at $\omega_1$ by showing that these assumptions are compatible with a failure of the $\Sigma_1$-colouring property at $\omega_2$. 
These arguments will also allow us to answer one of the main questions left open by the results of \cite{MR3694344} by showing that the existence of a $\Sigma_1(\omega_2)$-definable well-ordering of the reals is compatible with the existence of various very large cardinal assumptions (see {\cite[Question 7.5]{MR3694344}}). 
Finally, we will show that the $\mathbf{\Sigma}_2$-colouring  property provably fails for all successors of singular strong limits cardinals of uncountable cofinality.

In the proofs of the positive results mentioned above, we will often derive the following much stronger partition property for definable colourings.

\begin{definition}
 Let $\kappa$ be an uncountable regular cardinal.
  \begin{enumerate}[leftmargin=0.9cm]
   \item Given sets $z_0,\ldots,z_{m-1}$, the cardinal $\kappa$ has the \emph{$\Sigma_n(z_0,\ldots,z_{m-1})$-club property} if for every $\Sigma_n(\kappa,z_0,\ldots,z_{m-1})$-partition $\map{c}{[\kappa]^k}{\alpha}$ with $\alpha<\kappa$, there is a $c$-homogeneous set that is closed and unbounded in $\kappa$.  

 \item The cardinal $\kappa$ has the \emph{$\mathbf{\Sigma}_n$-club property} if it has the $\Sigma_n(z_0,\ldots,z_{m-1})$-club  property for all $z_0,\ldots,z_{m-1}\in\HH{\kappa}$.  
  \end{enumerate}
\end{definition}

For $n>0$, the $\mathbf{\Sigma}_n$-club property can easily be seen as a strengthening of the restriction of the partition property defining Ramsey cardinals to definable colourings, because, if $\map{c}{[\kappa]^{{<}\omega}}{\alpha}$ is a function with $\alpha<\kappa$ that is definable by a $\Sigma_n$-formula with parameter $z\in\HH{\kappa}$, then the  restrictions $c\restriction[\kappa]^k$ are all $\Sigma_n(z)$-partitions and hence this property yields a club in $\kappa$ that is $(c\restriction[\kappa]^k)$-homogeneous for all $k<\omega$. 
 We will present more justification for this view by showing that this implication also holds true when we consider alternative characterizations of Ramseyness through the existence of certain iterable models containing subsets of $\kappa$ and the restrictions of these properties to definable subsets. 
In fact, we will show that in the \emph{Dodd--Jensen core model $\KK^{DJ}$}, the $\mathbf{\Sigma}_1$-club property is equivalent to the restriction of Ramseyness to $\Sigma_1$-definable subsets of $\kappa$ in the above sense. 
In another direction, we will show that for all $n>0$, the validity of the $\mathbf{\Sigma}_n$-club property is equivalent to the non-existence of bistationary (i.e. stationary and costationary) subsets $A$ of $\kappa$ with the property that the corresponding set $\{A\}$ is definable by a $\Sigma_n$-formula with parameters in $\HH{\kappa}\cup\{\kappa\}$.

In the next section, we will show that all uncountable regular cardinals provably have the $\mathbf{\Sigma}_0$-club property and earlier remarks show that  this statement cannot be extended to $n=1$. Moreover, we will later show that the existence of a cardinal with the $\mathbf{\Sigma}_1$-club property implies the existence of $0^\#$. 
A short argument will allow us to show that a cardinal with the $\mathbf{\Sigma}_1$-club property is either equal to $\omega_1$ or a limit cardinal. 
Moreover, our results will show that many canonical extensions of $\ZFC$ cause $\omega_1$ to have the $\mathbf{\Sigma}_1$-club property, several large cardinal notions imply this property at the given large cardinal and the existence of an accessible regular limit cardinal with this property is consistent. 
Finally, we will show that no cardinal greater than $\omega_1$ has the $\mathbf{\Sigma}_2$-club property and that the statement that $\omega_1$ has the $\mathbf{\Sigma}_n$-club property for all $n<\omega$ is equiconsistent with the existence of a measurable cardinal.

We end this introduction by outlining the content of this paper. 
As a motivation for the later results of this paper, we show that all uncountable regular cardinals have the $\mathbf{\Sigma}_0$-club property in Section \ref{section:Sigma0}. 
In Section \ref{section:SigmaNpartiotionproperty}, we derive a long list of basic results on the $\mathbf{\Sigma}_n$-colouring property and present two alternative characterizations of this property that are also fragments of properties characterizing weakly compact cardinals. 
These results will allow us to determine the consistency strength of the $\mathbf{\Sigma}_n$-colouring property in many important cases. 
Section \ref{section:SigmaNclubproperty} contains an analogous  investigation of the $\mathbf{\Sigma}_n$-club property that provides the exact consistency strength of all instances of this property. 
In Section \ref{section:DefPartCountableOrdinals}, we use results from \cite{MR3694344} to show that both large cardinal assumptions and forcing axioms imply that $\omega_1$ has the $\mathbf{\Sigma}_1$-club property. 
In contrast, the results of Section \ref{section:Omega2} show that both of these assumptions are compatible with a failure of the $\Sigma_2$-colouring property at $\omega_2$. 
Section \ref{section:Limits} contains various example of non-weakly compact limit cardinals that provably have the $\mathbf{\Sigma}_1$-club property. 
In Section \ref{section:SuccSingular}, we will use results from \cite{DefSuccSingCardinals}, \cite{MR661475} and \cite{MR1462202} to study the $\mathbf{\Sigma}_n$-colouring property at successors of singular cardinals. 
Section \ref{section:DefHomeo} contains the results that originally motivated the work of this paper. These results deal with the question whether certain homeomorphisms witnessing failures of weak compactness can be simply definable and connect this question with the $\mathbf{\Sigma}_n$-colouring property. 
We conclude this paper in Section \ref{section:questions} with some question raised by its results.


\section{$\Sigma_0$-definable partitions}\label{section:Sigma0}

As a motivation for the main results of this paper, we show that all uncountable regular cardinals are weakly compact with respect to  $\Sigma_0$-definable colourings. In fact, we will prove to following stronger statement.

\begin{theorem}\label{theorem:Sigma0Partitions}
 Every uncountable regular cardinal has the $\mathbf{\Sigma}_0$-club property. 
\end{theorem}

In order to prove this results, we introduce equivalence relations on the classes of the form $[\On\setminus\xi]^{{<}\omega}$ that consist of all finite sets of ordinals greater than some fixed ordinal $\xi$.  
Given $0<l<\omega$ and $\xi\in\On$, we let $E_l^\xi$ denote the unique equivalence  relation on $[\On\setminus\xi]^{{<}\omega}$ such that for all $a,b\in[\On\setminus\xi]^{{<}\omega}$, we have $E_l^\xi(a,b)$ if and only if the following statements hold: 
  \begin{enumerate}[leftmargin=0.9cm]
   \item $\betrag{a}=\betrag{b}$. 
   
   \item Let $\alpha_1<\ldots<\alpha_k$ be the monotone enumeration of $a$ and let $\beta_1<\ldots<\beta_k$ is the monotone enumeration of $b$. Set $\alpha_0=\beta_0=\xi$. If  $i<k$, then there are ordinals $\mu$, $\nu$ and $\rho$ such that the following statements hold: 
    \begin{enumerate}
     \item $\alpha_{i+1}=\alpha_i+\omega^l\cdot\mu+\rho$.
     
     \item $\beta_{i+1}=\beta_i+\omega^l\cdot\nu+\rho$.
     
     \item $\rho<\omega^l$.
     
     \item $\min\{\mu,\nu\}=0$ implies $\mu=\nu=0$.   
    \end{enumerate}
  \end{enumerate} 
 Note that we have $E^\chi_{l+1}\subseteq E^\xi_l$ for all $0< l<\omega$ and $\xi\leq\chi\in\On$.

 \begin{proposition}\label{proposition:TupleExtensionEquiv}
  If $0<k,l<\omega$, $\xi\in\On$, $a\in[\On\setminus\xi]^{k+1}$, $b\in[\On\setminus\xi]^k$ and $\alpha\in a$ with $E^\xi_{l+1}(a\setminus\{\alpha\},b)$, then there is $\xi\leq\beta\in\On\setminus b$ with $E^\xi_l(a,b\cup\{\beta\})$. 
 \end{proposition}
 
 \begin{proof}
  Let $\alpha_1<\ldots<\alpha_k$ be the monotone enumeration of $a\setminus\{\alpha\}$, let $\beta_1<\ldots<\beta_k$ be the monotone enumeration of $b$ and set $\alpha_0=\beta_0=\xi$. 
  
  \smallskip
  
  \paragraph{\textbf{Case 1}: $\alpha=\xi$.} Pick $\mu$, $\nu$ and $\rho$ such that $\alpha_1=\xi+\omega^{l+1}\cdot\mu+\rho$ and $\beta_1=\xi+\omega^{l+1}\cdot\nu+\rho$. Since $\alpha\notin a$, we have $\alpha_1>\xi$ and either $\mu>0$ or $\rho>0$. This implies that $\beta_1>\xi$ and $\xi\notin b$. If we set $\beta=\xi$, then $E^\xi_{l+1}(a,b\cup\{\beta\})$ and therefore $E^\xi_l(a,b\cup\{\beta\})$. 
  
  \smallskip 
  
  \paragraph{\textbf{Case 2}: $\alpha>\alpha_k$.} Pick $\sigma,\tau\in\On$ with $\alpha=\alpha_k+\omega^l\cdot\sigma+\tau$ and $\tau<\omega^l$. If we set $\beta=\beta_k+\omega^l\cdot\sigma+\tau>\beta_k$, then $E^\xi_l(a,b\cup\{\beta\})$ holds. 
 
   \smallskip 
  
  \paragraph{\textbf{Case 3}: $\alpha_i<\alpha<\alpha_{i+1}$ for some $i\leq k$.} Pick $\mu,\nu,\rho\in\On$ such that $\alpha_{i+1}=\alpha_i+\omega^{l+1}\cdot\mu+\rho$, $\beta_{i+1}=\beta_i+\omega^{l+1}\cdot\nu+\rho$, $\rho<\omega^{l+1}$ and $\min\{\mu,\nu\}=0$ implies $\mu=\nu=0$.  
  
  \smallskip 
  
  \paragraph{\textbf{Subcase 3.1}: $\alpha\geq\alpha_i+\omega^{l+1}\cdot\mu$.} Pick $\sigma<\rho$ and $0<\tau\leq\rho$ with $\alpha=\alpha_i+\omega^{l+1}\cdot\mu+\sigma$ and $\rho=\sigma+\tau$. Set $\beta=\beta_i+\omega^{l+1}\cdot\nu+\sigma$. Then $\alpha_{i+1}=\alpha+\tau$ and $\beta_{i+1}=\beta+\tau>\beta$. This shows that  $E^\xi_{l+1}(a,b\cup\{\beta\})$ and therefore $E^\xi_l(a,b\cup\{\beta\})$.  
  
  \smallskip 
  
  \paragraph{\textbf{Subcase 3.2}: $\alpha<\alpha_i+\omega^{l+1}\cdot\mu$.} Then we can find $\mu_0<\mu$, $\mu_1\leq\mu$ and $\sigma<\omega^{l+1}$ with $\mu=\mu_0+1+\mu_1$ and $\alpha=\alpha_i+\omega^{l+1}\cdot\mu_0+\sigma$. Pick ordinals $\pi$ and $\tau$ such that $\sigma=\omega^l\cdot\pi+\tau$ and $\tau<\omega^l$. Then $\alpha=\alpha_i+\omega^l\cdot(\omega\cdot\mu_0+\pi)+\tau$ and 
  \begin{equation}\label{equation:Subcase321alphaRemainder}
   \begin{split}
    \alpha + \omega^{l+1}\cdot(1+\mu_1)+\rho  ~ & = ~  \alpha_i+\omega^{l+1}\cdot\mu_0+\sigma+ \omega^{l+1} + \omega^{l+1}\cdot\mu_1+\rho \\ 
    & = ~ \alpha_i+\omega^{l+1}\cdot(\mu_0+ 1 + \mu_1)+\rho ~ = ~  \alpha_{i+1},
   \end{split}
  \end{equation}
  because $\sigma<\omega^{l+1}$ implies that $\sigma+ \omega^{l+1}=\omega^{l+1}$. 
  
  \smallskip
  
  \paragraph{\textbf{Subcase 3.2.1}: $\omega\cdot\mu_0+\pi=0$.} Then $\tau>0$ and $\alpha=\alpha_i+\tau$. Set $\beta=\beta_i+\tau>\beta_i$. Since $\mu>0$ implies $\nu>0$, we have $\tau+\omega^{l+1}\cdot\nu=\omega^{l+1}\cdot\nu$, $$\beta+\omega^{l+1}\cdot\nu+\rho ~ = ~ \beta_i+\tau+\omega^{l+1}\cdot\nu+\rho ~ = ~ \beta_i+\omega^{l+1}\cdot\nu+\rho ~ = ~ \beta_{i+1}$$ and $\beta<\beta_{i+1}$. 
  In combination with (\ref{equation:Subcase321alphaRemainder}), this shows that  $E^\xi_{l+1}(a,b\cup\{\beta\})$ and we can conclude that $E^\xi_l(a,b\cup\{\beta\})$.

  \smallskip
  
    \paragraph{\textbf{Subcase 3.2.2}: $\omega\cdot\mu_0+\pi>0$.} Set $\beta=\beta_i+\omega^l+\tau$. Since $\omega^l+\tau<\omega^{l+1}$ and $\mu>0$ implies $\nu>0$, we then have  $$\beta+\omega^{l+1}\cdot\nu+\rho=\beta_i+\omega^l+\tau+\omega^{l+1}\cdot\nu+\rho=\beta_i+\omega^{l+1}\cdot\nu+\rho=\beta_{i+1}.$$ This allows us to conclude that $\beta_i<\beta<\beta_{i+1}$ and $E^\xi_l(a,b\cup\{\beta\})$ holds. 
 \end{proof}

We now use the above proposition to show that for all $\Sigma_0$-formulas, there are indices $l$ and $\xi$ such that the validity of the given formula is invariant across all $E^\xi_l$-equivalence classes.

 \begin{lemma}\label{lemma:EquivalentFormulasForEquivalentTuples}
  For every $\Sigma_0$-formula $\varphi(v_0,\ldots,v_K)$, every natural number $k\leq K$ and every injection $\map{\iota}{k+1}{K+1}$, there is a natural number $0<l_{\varphi,\iota}<\omega$ such that 
  $$\varphi(y_0,\ldots,y_K) ~ \longleftrightarrow ~ \varphi(z_0,\ldots,z_K)$$ holds for all sets $y_0,\ldots,y_K,z_0,\ldots,z_K$ such that there are $\xi\in\On$ and $a,b\in[\On\setminus\xi]^{k+1}$ satisfying the following statements: 
  \begin{enumerate}[leftmargin=0.9cm]
  \item $E^\xi_{l_{\varphi,\iota}}(a,b)$.
  
   \item If $\alpha_0<\ldots<\alpha_k$ is the monotone enumeration of $a$ and $\beta_0<\ldots<\beta_k$ is the monotone enumeration of $b$, then $\alpha_i=y_{\iota(i)}$ and $\beta_i=z_{\iota(i)}$ for all $i\leq k$. 
   
   \item If $j\leq K\setminus\ran{\iota}$, then $y_j=z_j$ and $\tc{\{y_j\}}\cap\On\subseteq\xi$. 
  \end{enumerate}
 \end{lemma}
 
 \begin{proof}
  We prove the above statement by induction on the complexity of $\varphi$. 
  
  First, assume that $\varphi$ is atomic and set $l_{\varphi,\iota}=1$. Then some easy case distinctions show that the above assumptions (ii) and (iii) imply the desired equivalence for $\varphi$. In the case of negations and conjunctions, the above statement follows directly from the induction hypothesis if we set $l_{\neg\varphi,\iota}=l_{\varphi,\iota}$ and $l_{\varphi_0\wedge\varphi_1,\iota}=\max\{l_{\varphi_0,\iota},l_{\varphi_1,\iota}\}$. Finally, assume that $\varphi\equiv\exists x\in v_j  ~ \psi(v_0,\ldots,v_K,x)$ and the above statement holds for $\psi(v_0,\ldots,v_{K+1})$. Given $i\leq k+1$, let $\map{\tau_i}{i+1}{k+2}$ denote the unique order-preserving function with $i\notin\ran{\tau_i}$ and let $\map{\iota_i}{k+2}{K+2}$ denote the unique injection with $\iota_i(i)=K+1$ and $\iota(h)=(\iota_i\circ\tau_i)(h)$ for all $h\leq k$. 
Next, given $i\leq k$, let $\psi_i(v_0,\ldots,v_K)$ denote the formula obtained from $\psi$ by replacing all occurrences of the variable $v_{K+1}$ with the variable $v_{\iota(i)}$. Define  
$$l_{\varphi,\iota} ~ = ~ \max\{l_{\psi,\iota},l_{\psi_0,\iota},\ldots,l_{\psi_k,\iota},l_{\psi,\iota_0}+1,\ldots,l_{\psi,\iota_{k+1}}+1\}$$ and fix sets $y_0,\ldots,y_K,z_0,\ldots,z_K$, an ordinal $\xi$ and sets $a,b\in[\On\setminus\xi]^{k+1}$ that satisfy the above statements (i)-(iii) with respect to $\iota$ and $l_{\varphi,\iota}$.  Now, assume that there is an $y_{N+1}\in y_j$ such that $\psi(y_0,\ldots,y_{K+1})$ holds. First, if either $j\notin\ran{\iota}$ or $y_{K+1}\in\xi$, then we know that $E^\xi_{l_{\psi,\iota}}(a,b)$, $K+1\notin\ran{\iota}$ and $\tc{\{y_{K+1}\}}\cap\On\subseteq\xi$. Therefore our induction hypothesis implies that $\psi(z_0,\ldots,z_K,y_{K+1})$ holds in this case. 
Next, if $y_{K+1}=\alpha_i$ for some $i\leq k$, then $E^\xi_{l_{\psi_i},\iota}(a,b)$  and our induction hypothesis implies that $\psi(z_0,\ldots,z_K,\beta_i)$ holds. 
Finally, assume that  $j\in\ran{\iota}$, $\xi\leq y_{K+1}\notin a$ and $y_{K+1}$ is the $i$-th element in the monotone enumeration of $a\cup\{y_{K+1}\}$. Then $E^\xi_{l_{\psi,\iota_i}+1}(a,b)$ and  Proposition \ref{proposition:TupleExtensionEquiv} yields a  $\xi\leq\beta_{k+1}\in\On\setminus b$ with $E^\xi_{l_{\psi,\iota_i}}(a\cup\{y_{K+1}\},b\cup\{\beta_{k+1}\})$. In particular, our induction hypothesis implies that $\psi(z_0,\ldots,z_K,\beta_{k+1})$ holds. 
In all of the above cases, we can conclude that $\varphi(y_0,\ldots,y_K)$ implies  $\varphi(z_0,\ldots,z_K)$.  Moreover, the same arguments show that $\varphi(z_0,\ldots,z_K)$ also implies $\varphi(y_0,\ldots,y_K)$. 
 \end{proof}

\begin{proof}[Proof of Theorem \ref{theorem:Sigma0Partitions}]
  Let $\kappa$ be an uncountable regular cardinal, let $z$ be an element of $\HH{\kappa}$ and let $\map{c}{[\kappa]^k}{\alpha}$ be a $\Sigma_0(z)$-partition with $\alpha<\kappa$. 
Then there is a $\Sigma_0$-formula $\varphi(v_0,\ldots,v_{k+2})$ with the property that for  $\alpha_0<\ldots<\alpha_{k-1}<\kappa$, $c(\{\alpha_0,\ldots,\alpha_{k-1}\})$ is the unique ordinal $\gamma$ such that  $\varphi(\alpha_0,\ldots,\alpha_{k-1},\kappa,\gamma,z)$ holds.  
Pick an ordinal $\alpha+\omega^\omega<\xi<\kappa$ with $\tc{\{z\}}\cap\On\subseteq\xi$, let $H$ be the set of all  multiplicatively indecomposable ordinals in the interval $[\xi,\kappa]$ and let $\iota$ denote the identity function on $k+1$. 
Then $\kappa\in H$, $C=H\cap\kappa$ is a club in $\kappa$ and $E^\xi_{l_{\varphi,\iota}}(a,b)$ holds for all $a,b\in[H]^{n+1}$.
But then Lemma \ref{lemma:EquivalentFormulasForEquivalentTuples} shows that, if $\alpha_0<\ldots<\alpha_{k-1}$ is the monotone enumeration of $a\in[C]^k$ and $\beta_0<\ldots<\beta_{k-1}$ is the monotone enumeration of $b\in[C]^k$, then $$\varphi(\alpha_0,\ldots,\alpha_{k-1},\kappa,\gamma,z) ~ \longleftrightarrow ~ \varphi(\beta_0,\ldots,\beta_{k-1},\kappa,\gamma,z)$$ for all $\gamma<\alpha$ and therefore $c(a)=c(b)$. 
\end{proof}


\section{The $\Sigma_n$-colouring property}\label{section:SigmaNpartiotionproperty}

In the remainder of this paper, we always use $n$ to denote a natural number greater than $0$. Note that, since sets of the form $\HH{\kappa}$ are closed under the pairing functions, this assumptions allows us to only consider $\Sigma_n$-formulas that use a single parameter from $\HH{\kappa}$ when we verify that an uncountable regular cardinal $\kappa$ has the $\mathbf{\Sigma}_n$-colouring property.

This section contains a number of basic results about the $\mathbf{\Sigma}_n$-colouring property that generalize fundamental results about weakly compact cardinals to the definable setting. 
These results will allow us to show that for all $0<n<\omega$, there is a natural connection between the $\mathbf{\Sigma}_n$-colouring property and a large cardinal property, in the sense that the large cardinal implies the $\mathbf{\Sigma}_n$-colouring property, the $\mathbf{\Sigma}_n$-colouring property implies that the given cardinal has the large cardinal property in $\LL$ and it is possible to use forcing to turn an inaccessible cardinal with the relevant large cardinal property into either the successor of a regular cardinal or into an accessible regular limit cardinal with the $\mathbf{\Sigma}_n$-colouring property. 
For $n\geq 2$, the corresponding large cardinal property will turn out to be weak compactness. 
In contrast, our results will show that the $\mathbf{\Sigma}_1$-colouring  property corresponds to a large cardinal property strictly between Mahloness and weak compactness. 
Finally, our results will also allow us present several ways to establish the consistency of failures of definable weak compactness.

The following result transfers the fact that weakly compact cardinals are inaccessible to the definable setting.

\begin{proposition}\label{Proposition:NoDefinableInjection}
 Let $\kappa$ be an uncountable regular cardinal with the $\Sigma_n(z)$-colouring property. If $\map{f}{\kappa}{{}^{{<}\kappa}2}$ is a $\Sigma_n(\kappa,z)$-definable function and $\gamma<\kappa$, then the set $\Set{f(\alpha)\restriction\gamma}{\alpha<\kappa}$ has cardinality less than $\kappa$. 
\end{proposition}

\begin{proof}
 Assume, towards a contradiction, that there is a $\gamma<\kappa$ with the property that the set $\Set{f(\alpha)\restriction\gamma}{\alpha<\kappa}$ has cardinality $\kappa$. 
 Let $\delta$ be minimal with this property. 
Then it is easy to see that the set $\{\delta\}$ is $\Sigma_n(\kappa,z)$-definable and the minimality of $\delta$ implies that the set $\Set{f(\alpha)\restriction\delta}{\alpha<\kappa, ~ \delta\subseteq\dom{f(\alpha)}}$ also has cardinality $\kappa$. 
Let $\map{i}{\kappa}{\kappa}$ be the unique injection with the property that for all $\alpha<\kappa$, the image $i(\alpha)$ is the minimal $\beta<\kappa$ with $\delta\subseteq\dom{f(\beta)}$ and $f(\beta)\restriction\delta\neq f(i(\bar{\alpha}))\restriction\delta$ for all $\bar{\alpha}<\alpha$. 
 Then the $\Sigma_n$-Recursion Theorem implies that  $i$ is $\Sigma_n(\kappa,z)$-definable and this shows that the injection $$\Map{\iota}{\kappa}{{}^\delta 2}{\alpha}{(f\circ i)(\alpha)\restriction\delta}$$ is definable in the same way. 
 Set $$\Delta(\alpha,\beta) ~ = ~ \min\Set{\gamma<\delta}{\iota(\alpha)(\gamma)\neq\iota(\beta)(\gamma)}$$ for all $\alpha<\beta<\kappa$ and let $\map{c}{[\kappa]^2}{2}$ denote the unique map satisfying $$c(\{\alpha,\beta\})=0 ~ \Longleftrightarrow ~ \iota(\alpha)(\Delta(\alpha,\beta))<\iota(\beta)(\Delta(\alpha,\beta))$$ for all $\alpha<\beta<\kappa$. Then $c$ is $\Sigma_n(\kappa,z)$-definable and our assumptions yield a $c$-homogeneous set $H$ that is unbounded in $\kappa$. 
 Given $\gamma<\delta$, let $H_\gamma$ denote the set of all $\alpha\in H$ with the property that $\gamma$ is the minimal element of $\delta$ with $\gamma=\Delta(\alpha,\beta)$ for some $\alpha<\beta\in H$. 
  Since $H=\bigcup\Set{H_\gamma}{\gamma<\delta}$, there is a $\gamma_*<\delta$ with $H_{\gamma_*}$ unbounded in $\kappa$. 
Fix $\alpha_0,\alpha_1\in H_{\gamma_*}$ and $\beta_0,\beta_1\in H$ with $\alpha_0<\beta_0<\alpha_1<\beta_1$ and $\gamma_*=\Delta(\alpha_0,\beta_0)=\Delta(\alpha_1,\beta_1)$. 
 Then the minimality of $\gamma_*$ implies that $$\iota(\alpha_0)\restriction\gamma_* ~ = ~ \iota(\beta_0)\restriction\gamma_* ~ = ~ \iota(\alpha_1)\restriction\gamma_* ~ = ~ \iota(\beta_1)\restriction\gamma_*$$ and therefore $\iota(\beta_0)(\gamma_*)=\iota(\alpha_1)(\gamma_*)$, because otherwise we would have $\Delta(\alpha_0,\beta_0)=\Delta(\beta_0,\alpha_1)=\gamma_*$ and the homogeneity of $H$ would imply that the ordinals $\iota(\alpha_0)(\gamma_*)$, $\iota(\beta_0)(\gamma_*)$ and $\iota(\alpha_1)(\gamma_*)$ are pairwise different. 
 But then $\Delta(\alpha_0,\beta_0)=\Delta(\beta_0,\beta_1)=\gamma_*$ and this allows us to conclude that the ordinals $\iota(\alpha_0)(\gamma_*)$, $\iota(\beta_0)(\gamma_*)$ and $\iota(\alpha_1)(\gamma_*)$ are pairwise different, a contradiction.  
\end{proof}

\begin{corollary}\label{corollary:InaccessibleInLA}
 If $\kappa$ is an uncountable regular cardinal with the $\Sigma_n(z)$-colouring property and $A$ is a subset of $\kappa$ such that the set $\{A\}$ is $\Sigma_n(\kappa,z)$-definable, then $\kappa$ is inaccessible in $\LL[A]$. 
\end{corollary}

\begin{proof}
 Assume that the above conclusion fails. 
Let $\iota$ denote the $<_{\LL[A]}$-least injection of $\kappa$ into some ${}^\nu 2$ with $\nu<\kappa$ in $\LL[A]$. 
 By our assumptions, the sets $\{\nu\}$ and $\{\iota\}$ are both $\Sigma_n(\kappa,z)$-definable and hence there is a $\Sigma_n(\kappa,z)$-definable injection from $\kappa$ into ${}^\nu 2$, contradicting Proposition \ref{Proposition:NoDefinableInjection}.  
\end{proof}

Proposition \ref{Proposition:NoDefinableInjection} also allows us to show that a small partial order can force a failure of the $\mathbf{\Sigma}_1$-colouring property at the successors of an  uncountable regular cardinal.   
In particular, large cardinal axioms do not imply that successors of uncountable regular cardinals have the $\mathbf{\Sigma}_1$-colouring property. 
The results of Section \ref{section:DefPartCountableOrdinals} will show that the situation for $\omega_1$ is completely different.

\begin{corollary}\label{corollary:GenericTreeCodingCounterexamples}
 If $\nu$ is an uncountable cardinal with $\nu=\nu^{{<}\nu}$, then there is a ${<}\nu$-closed partial order $\PPP$ satisfying the $\nu^+$-chain condition with 
 \begin{equation}\label{equation:FailureSigma1PP}
  \mathbbm{1}_\PPP\Vdash\anf{\textit{The cardinal $\nu^+$ does not have the $\mathbf{\Sigma}_1$-colouring property}}.
 \end{equation}
\end{corollary}

\begin{proof}
 Fix an injection $\map{\iota}{\nu^+}{{}^\nu 2}$ and set $A=\Set{\langle\iota(\alpha),\iota(\beta)\rangle}{\alpha<\beta<\nu^+}$. 
By {\cite[Theorem 1.5]{MR2987148}}, there is a ${<}\nu$-closed partial order $\PPP$ satisfying the $\nu^+$-chain condition with the property that whenever $G$ is $\PPP$-generic over $\VV$, then there is $z\in\POT{\nu}^{\VV[G]}$ such that the set $A$ is $\Sigma_1(\nu,z)$-definable in $\VV[G]$. 
 But then Proposition \ref{Proposition:NoDefinableInjection} shows that (\ref{equation:FailureSigma1PP}) holds. 
\end{proof}

The next lemma now generalizes the characterizations of weak compactness through the tree property (see {\cite[Theorem 7.8]{MR1994835}}) and certain  elementary embeddings (see \cite{MR1133077}). Remember that, given an infinite cardinal $\kappa$, a \emph{weak $\kappa$-model} is a transitive model $M$ of $\mathsf{ZFC}^-$ of size $\kappa$ with $\kappa\in M$.\footnote{By $\ZFC^-$, we mean the usual axioms of $\ZFC$ without the power set axiom, however including the Collection scheme instead of the Replacement scheme. Note that $\HH{\kappa}$ is a model of this
theory for every uncountable regular cardinal $\kappa$. }

\begin{lemma}\label{lemma:CharacterizationsPartitionProperty}
 The following statements are equivalent for every uncountable regular cardinal $\kappa$ and every set $z$: 
 \begin{enumerate}[leftmargin=0.9cm]
  \item $\kappa$ has the $\Sigma_n(z)$-colouring property.

  \item If $\map{\iota}{\kappa}{{}^{{<}\kappa}2}$ is a $\Sigma_n(\kappa,z)$-definable injection, then there is an $x\in{}^\kappa 2$ with the property that the set $\Set{\alpha<\kappa}{\exists\beta<\kappa ~ x\restriction\alpha\subseteq\iota(\beta)}$ is unbounded in $\kappa$. 
  
    \item If $A\subseteq\kappa$ with the property that $\{A\}$ is $\Sigma_n(\kappa,z)$-definable, then there is a weak $\kappa$-model $M$, a transitive set $N$ and an elementary embedding $\map{j}{M}{N}$ such that $A\in M$, $\crit{j}=\kappa$, $\kappa$ is inaccessible in $M$ and $\HH{\kappa}^M\in M$.  
 \end{enumerate}
 \end{lemma}

 \begin{proof}
 Assume that (i) holds and let $\map{\iota}{\kappa}{{}^{{<}\kappa}2}$ be a $\Sigma_n(\kappa,z)$-definable injection. Remember that the lexicographic ordering $<_{lex}$ of ${}^{{<}\kappa}2$ is the unique linear ordering  of ${}^{{<}\kappa}2$ with the property that for all $s,t\in{}^{{<}\kappa}2$, we have $s<_{lex}t$ if either $s\subsetneq t$ or there is an ordinal $\alpha\in\dom{s}\cap\dom{t}$ with $s\restriction\alpha=t\restriction\alpha$ and $s(\alpha)<t(\alpha)$. Given $s,t,u,v\in{}^{{<}\kappa}2$ with $s\subseteq t\cap v$ and $t<_{lex} u<_{x} v$, a short computation shows that $s\subseteq u$ holds.
  Let $\map{c}{[\kappa]^2}{2}$ denote the unique function with the property that for all $\alpha<\beta<\kappa$, we have $c(\alpha,\beta)=0$ if and only if $\iota(\alpha)<_{lex}\iota(\beta)$. 
  Then $c$ is a $\Sigma_n(z)$-partition and our assumption yields a $c$-homogeneous subset $H$ of $\kappa$ of cardinality $\kappa$.

 \begin{claim*}
  Given $\gamma<\kappa$, there is $\gamma<\alpha_\gamma\in H$ and  $t_\gamma\in{}^\gamma 2$ with $t_\gamma\subseteq\iota(\alpha)$ for all $\alpha_\gamma<\alpha\in H$. 
 \end{claim*}

 \begin{proof}[Proof of the Claim]
  By Proposition \ref{Proposition:NoDefinableInjection}, there is a sequence  $t_\gamma\in{}^\gamma 2$ with the property that the set $H_\gamma=\Set{\alpha\in H}{t_\gamma\subseteq\iota(\alpha)}$ has cardinality $\kappa$. 
  Define $\alpha_\gamma=\min(H_\gamma)$, fix $\alpha_\gamma<\alpha\in H$ and pick $\alpha<\beta\in H_\gamma$. Then $t_\gamma\subseteq\iota(\alpha_\gamma)\cap\iota(\beta)$ and we  either have $\iota(\alpha_\gamma)<_{lex}\iota(\alpha)<_{lex}\iota(\beta)$ or $\iota(\beta)<_{lex}\iota(\alpha)<_{lex}\iota(\alpha_\gamma)$. By the above remarks, we can conclude  that $t_\gamma\subseteq\iota(\alpha)$. 
 \end{proof}

 Pick $\gamma<\delta<\kappa$ and $\max\{\alpha_\gamma,\alpha_\delta\}<\alpha\in H$. Then the above claim yields  $t_\gamma\subseteq t_\delta\subseteq\iota(\alpha)$ and this implies that $x=\bigcup\Set{t_\gamma}{\gamma<\kappa}$ is an element of ${}^\kappa 2$ with the property that the set $\Set{\alpha<\kappa}{\exists\beta<\kappa ~ x\restriction\alpha\subseteq\iota(\beta)}$ is unbounded in $\kappa$.


  Next, assume that (ii) holds and let $\map{c}{[\kappa]^2}{2}$ be a $\Sigma_n(z)$-partition.  
  Then the proof of the classical \emph{Ramification Lemma} (see, for example, {\cite[Lemma 7.2]{MR1994835}}) yields a unique sequence $\seq{<_\alpha}{\alpha<\kappa}$ such that the following statements hold for all $\alpha<\kappa$:  
 \begin{enumerate}[leftmargin=0.9cm]
  \item[(a)] $<_\alpha$ is a binary relation on $\alpha$ that extends the $\in$-relation, $\langle\alpha,<_\alpha\rangle$ is a tree  and, if $\alpha$ is a limit ordinal, then ${<_\alpha}={\bigcup\Set{<_{\bar{\alpha}}}{\bar{\alpha}<\alpha}}$.

  \item[(b)] We have $0<_2 1$ and, if $\alpha>1$, then there is a unique maximal branch $b_\alpha$ through $\langle\alpha,<_\alpha\rangle$ with  $c(\alpha_0,\alpha_1)=c(\alpha_0,\alpha)$  for all $\alpha_0,\alpha_1\in b_\alpha$ satisfying  $\alpha_0<\alpha_1$.

  \item[(c)] Given $\alpha<\beta<\kappa$, we have ${<_\alpha}={{<_\beta}\restriction(\beta\times\alpha)}$ and, if $\alpha>1$, then $b_\alpha$ is equal to the set of all predecessors of $\alpha$ in $\langle\beta,<_\beta\rangle$. 
 \end{enumerate}

Then there is a unique binary relation $<_c$ on $\kappa$ with ${<_c}={\bigcup\Set{<_\alpha}{\alpha<\kappa}}$ for some sequence $\seq{<_\alpha}{\alpha<\kappa}$ with the above properties and the structure $\langle\kappa,<_c\rangle$ is a tree. 
Moreover, the uniqueness of the sequence $\seq{<_\alpha}{\alpha<\kappa}$ and the branches $\seq{b_\alpha}{1<\alpha<\kappa}$ implies that the  set $\{<_c\}$ is $\Sigma_n(\kappa,z)$-definable. 

 \begin{claim*} 
  Every $\alpha<\kappa$ has at most two direct successors in $\langle\kappa,<_c\rangle$. 
 \end{claim*}

 \begin{proof}[Proof of the Claim]
  Otherwise, we can find  
 $\alpha<\beta_0<\beta_1<\kappa$ such that $\beta_0$ and $\beta_1$ are both direct successors of $\alpha$ in $\langle\kappa,<_c\rangle$ and 
$c(\alpha,\beta_0)=c(\alpha,\beta_1)$. Since $b_{\beta_0}=b_{\beta_1}=b_\alpha\cup\{\alpha\}$, our assumptions imply that $c(\alpha_0,\alpha_1)=c(\alpha_0,\beta_1)$ holds for all $\alpha_0,\alpha_1\in b_{\beta_1}\cup\{\beta_0\}$. But this contradicts the maximality of $b_{\beta_1}$.  
 \end{proof}

Now let $\map{\iota}{\kappa}{{}^{{<}\kappa}2}$ denote the unique injection with $\dom{\iota(\beta)}=\beta+1$ and $$\iota(\beta)(\alpha)=1 ~ \Longleftrightarrow ~ \left( \alpha<_c\beta ~ \vee ~ \alpha=\beta\right)$$ for all $\alpha\leq\beta<\kappa$. Then $\iota$ is $\Sigma_n(\kappa,z)$-definable and our assumption (ii) yields $x\in{}^\kappa 2$ with the property  that the set $\Set{\alpha<\kappa}{\exists\beta<\kappa ~ x\restriction\alpha\subseteq\iota(\beta)}$ is unbounded in $\kappa$. Define $K=\Set{\alpha<\kappa}{x(\alpha)=1}$. Then $\alpha<_c\beta$ for all $\alpha,\beta\in K$ with $\alpha<\beta$.

 \begin{claim*} 
  The set $K$ is unbounded in $\kappa$. 
 \end{claim*}

 \begin{proof}[Proof of the Claim]
  First, assume that $K$ has a maximal element $\alpha<\kappa$. Then the above claim shows that there is a $\beta\in\Lim\cap\kappa$ such that all direct successor of $\alpha$ in $\langle\kappa,<_c\rangle$ are elements of $\beta$.  
Pick $\gamma<\kappa$ that is minimal with the property that $x\restriction\beta\subseteq\iota(\gamma)$. 
Then $\beta\leq\gamma$ and  then the minimality of $\gamma$ implies that $\iota(\gamma)(\bar{\gamma})=0=x(\bar{\gamma})$ for all $\beta\leq\bar{\gamma}<\gamma$, because otherwise $\iota(\gamma)(\bar{\gamma})=1$ would imply that $\iota(\bar{\gamma})=\iota(\gamma)\restriction(\bar{\gamma}+1)$ and hence $x\restriction\beta\subseteq\iota(\bar{\gamma})$. 
This shows that $\iota(\gamma)\restriction\gamma=x\restriction\gamma$. 
Since $\iota(\gamma)(\alpha)=x(\alpha)=1$, we can conclude that $\gamma$ is a direct successor of $\alpha$ in $\langle\kappa,<_c\rangle$ that is not contained in $\beta$, a contradiction.

Now, assume that $K$ is a cofinal subset of $\alpha\in\Lim\cap\kappa$. Pick $\beta_0<\kappa$ minimal
with $x\restriction\alpha\subseteq\iota(\beta_0)$. 
Then $\alpha\leq\beta_0$ and, since $K\subseteq\alpha$, the minimality of $\beta_0$ implies that $x\restriction\beta_0=\iota(\beta_0)\restriction\beta_0$.  
Next, pick $\beta_1<\kappa$ minimal with $x\restriction(\beta_0+1)\subseteq\iota(\beta_1)$. 
Then $\beta_0\leq\beta_1$ and $x(\beta_0)=0<1=\iota(\beta_1)(\beta_1)$ implies that $\beta_0<\beta_1$. Then the minimality of $\beta_1$ and $K\subseteq\alpha$ imply  $x\restriction\beta_1=\iota(\beta_1)\restriction\beta_1$.  
 In particular, we have $b_{\beta_0}=K=b_{\beta_1}$. 
  Given $\alpha_0\in K$, there is $\alpha_1\in K$ with $\alpha_0<\alpha_1$ and the above equalities imply that $c(\alpha_0,\beta_0)=c(\alpha_0,\alpha_1)=c(\alpha_0,\beta_1)$. This shows that $c(\alpha_0,\alpha_1)=c(\alpha_0,\beta_1)$ holds for all $\alpha_0,\alpha_1\in b_{\beta_1}\cup\{\beta_0\}$, contradicting the maximality of $b_{\beta_1}$. 
 \end{proof}

 If we define $$\Map{f}{K}{2}{\alpha}{c(\alpha,\min(K\setminus(\alpha+1)))},$$ then the above claim yields an unbounded subset $H$ of $K$ with $f\restriction H$ is constant.  
 Since $\alpha<_c\beta$ holds for all $\alpha,\beta\in K$ with $\alpha<\beta$, we know that $c(\alpha,\beta)=c(\alpha,\gamma)$ for all $\alpha,\beta,\gamma\in K$ with $\alpha<\beta\leq\gamma$. 
 In particular, if $\alpha,\beta\in H$ with $\alpha<\beta$, then $c(\alpha,\beta)=c(\alpha,\min(K\setminus(\alpha+1)))=f(\alpha)$. This shows that $H$ is $c$-homogeneous.


 Now, assume that (ii) holds and pick $A\subseteq\kappa$ such that the set $\{A\}$ is $\Sigma_n(\kappa,z)$-definable. 
Since we know that (i) holds, we can use Corollary \ref{corollary:InaccessibleInLA} to show that $\kappa$ is inaccessible in $\LL[A]$ and hence $({}^{{<}\kappa}2)^{\LL[A]}\subseteq\LL_\kappa[A]=\HH{\kappa}^{\LL[A]}$.  
Let $\theta>\kappa$ be minimal with $\LL_\theta[A]\models\ZFC^-+\anf{\textit{$\POT{\kappa}$ exists}}$, let $b$ be the $<_{\LL[A]}$-minimal bijection between $\kappa$ and $\POT{\kappa}^{\LL_\theta[A]}$ in $\LL[A]$ and let $\vartheta>\theta$ be minimal with the property that   $b\in\LL_\vartheta[A]\models\ZFC^-+\anf{\textit{$\POT{\kappa}$ exists}}$. 
Then the sets $\{\LL_\theta[A]\}$, $\{\LL_\vartheta[A]\}$ and $\{b\}$ are all $\Sigma_1(\kappa,A)$-definable and therefore our assumption implies that they are also $\Sigma_n(\kappa,z)$-definable. 
 Define $$B_t ~ = ~ \left(\bigcap\Set{b(\alpha)}{t(\alpha)=1}\right) ~ \cap ~ \left(\bigcap\Set{\kappa\setminus b(\alpha)}{t(\alpha)=0}\right) ~ \in ~ \POT{\kappa}^{\LL_\vartheta[A]}$$ for all $t\in({}^{{<}\kappa}2)^{\LL[A]}$ and let $\calB$ be the set of all $t\in({}^{{<}\kappa}2)^{\LL[A]}$ with $\betrag{B_t}^{\LL_\vartheta[A]}=\kappa$.

\begin{claim*}
 The set $\calB$ has cardinality $\kappa$.
\end{claim*}
 
 \begin{proof}[Proof of the Claim]
  Assume not. Then there is a minimal $\beta<\kappa$ with $\calB\subseteq{}^{{<}\beta}2$.  Let $\map{f}{\kappa}{{}^\beta 2}$ denote the unique function with $$f(\gamma)(\alpha)=1 ~ \Longleftrightarrow ~ \alpha\in b(\gamma)$$ for all $\gamma<\kappa$ and $\alpha<\beta$. 
Then $\ran{f}\subseteq\LL[A]$ and for all $\gamma<\kappa$, we have $\gamma\in B_{f(\gamma)}$ and $\betrag{B_{f(\gamma)}}<\kappa$. This shows that $\betrag{\ran{f}}=\kappa$. 
Let $\map{\iota}{\kappa}{\ran{f}}$ denote the monotone enumeration of $\ran{f}$ with respect to $<_{\LL[A]}$. 
Then the set $\{\beta\}$, the function $f$ and the function $\iota$ are all definable over the structure $\langle\LL_\vartheta[A],\in\rangle$ by a formula with parameters $A$ and $b$. 
But this shows that $\iota$ is a $\Sigma_n(\kappa,z)$-definable injection from $\kappa$ into ${}^\beta 2$, contradicting Proposition \ref{Proposition:NoDefinableInjection}.  
 \end{proof}

 Now, let $\map{\iota}{\kappa}{{}^{{<}\kappa}2}$ denote the monotone enumeration of $\calB$ with respect to $<_{\LL[A]}$. As above, we know that $\iota$ is $\Sigma_n(\kappa,z)$-definable and hence the assumption (ii) yields an $x\in{}^\kappa 2$ with $\Set{\alpha<\kappa}{\exists\beta<\kappa ~ x\restriction\alpha\subseteq\iota(\beta)}$ is unbounded in $\kappa$.  If we define $$U ~ = ~ \Set{B\subseteq\kappa}{\exists\gamma<\kappa ~ B_{x\restriction\gamma}\subseteq B},$$ then it is easy to see that $U$ is a non-principal, ${<}\kappa$-complete filter on $\kappa$ that measures every subset of $\kappa$ contained in $\LL_\theta[A]$.  In particular, this implies that the  ultrapower $\Ult{\LL_\theta[A]}{U\cap\POT{\kappa}^{\LL_\theta[A]}}$ (that uses only functions $\map{f}{\kappa}{\LL_\theta[A]}$ contained in $\LL_\theta[A]$) is well-founded and, if we let $N$ denote its transitive collapse, then the  corresponding elementary embedding $\map{j}{\LL_\theta[A]}{N}$ has critical point $\kappa$.

 Finally, assume (iii) and let $\map{\iota}{\kappa}{{}^{{<}\kappa}2}$ be a $\Sigma_n(\kappa,z)$-definable injection. Set\footnote{We let use $\goedel{\cdot}{\ldots,\cdot}$ to denote  (iterated applications of) the G\"odel pairing function.} $$A ~ = ~ \Set{\goedel{\alpha}{\gamma,\iota(\alpha)(\gamma)}}{\alpha<\kappa, ~ \gamma\in\dom{\iota(\gamma)}}.$$ 
Then the set $\{A\}$ is also $\Sigma_n(\kappa,z)$-definable and (iii) yields a weak $\kappa$-model $M$, a transitive set $N$ and an elementary embedding $\map{j}{M}{N}$ such that  $\crit{j}=\kappa$, $A\in M$ and $\kappa$ is inaccessible in $M$.  
 Since $\kappa$ is inaccessible in $M$ and $\HH{\kappa}^M\in M$, elementarity implies that $\HH{\kappa}^N\subseteq M$. 
 Moreover, $A\in M$ implies that $\iota$ is an element of $M$. 
Set $t=j(\iota)(\kappa)$ and assume, towards a contradiction, that $\dom{t}<\kappa$. Then $t=j(t)\in\HH{\kappa}^N\subseteq M$ and elementarity yields an $\alpha<\kappa$ with $\iota(\alpha)=t$. But then $j(\iota)(\alpha)=t=j(\iota)(\kappa)$, a contradiction. This shows that $\dom{t}\geq\kappa$ and $x=t\restriction\kappa\in{}^\kappa 2$. If $\gamma<\kappa$, then $x\restriction\gamma\in M$ and elementarity yields an $\alpha<\kappa$ with $x\restriction\gamma\subseteq\iota(\alpha)$. Therefore $x$ witnesses that (ii) holds with respect to $\iota$.  
\end{proof}

We now use the above characterizations to strengthen the conclusion of Corollary \ref{corollary:InaccessibleInLA} and isolate the large cardinal properties that correspond to the $\mathbf{\Sigma}_n$-colouring properties.

\begin{corollary}
 If $\kappa$ is an uncountable regular cardinal with the $\Sigma_n(z)$-colouring property and $A$ is a subset of $\kappa$ with the property that the set $\{A\}$ is $\Sigma_n(\kappa,z)$-definable, then $\kappa$ is a Mahlo cardinal in $\LL[A]$. 
\end{corollary}

\begin{proof}
 Assume that the above conclusion fails. 
Since Corollary \ref{corollary:InaccessibleInLA} implies that $\kappa$ is inaccessible in $\LL[A]$, a result of Todor\v{c}evi\'{c} (see {\cite[Theorem 6.1.4]{MR2355670}}) shows that $\LL[A]$ contains a special $\kappa$-Aronszajn tree (see {\cite[Definition 6.1.1]{MR2355670}}). 
By using $<_{\LL[A]}$ and $\goedel{\cdot}{\cdot}$ to code the $<_{\LL[A]}$-least special $\kappa$-Aronszajn tree in $\LL[A]$ into an element of $\POT{\kappa}^{\LL[A]}$, we find $B\subseteq\kappa$ with the property that the set $\{B\}$ is $\Sigma_1(\kappa,A)$-definable and every weak $\kappa$-model that contains $B$ also contains a special $\kappa$-Aronszajn tree. 
Then the set $\{B\}$ is $\Sigma_n(\kappa,z)$-definable and Lemma \ref{lemma:CharacterizationsPartitionProperty} yields a weak $\kappa$-model $M$, a transitive set $N$ and an elementary embedding $\map{j}{M}{N}$ with $\crit{j}=\kappa$ and $B\in M$.  Then $M$ contains a special $\kappa$-Aronszajn tree $\TTT$ and every element of the $\kappa$-th level of $j(\TTT)$ in $N$ induces a cofinal branch through $\TTT$. Since $\TTT$ is special, this contradicts the regularity of $\kappa$. 
\end{proof}

\begin{corollary}\label{corollary:PartitionPropertyDownwardsL}
 Let $\kappa$ be an uncountable regular cardinal, let $x\in\HH{\kappa}\cap\POT{\kappa}$ and let $z\in\HH{\kappa^+}^{\LL[x]}$. 
 If $\kappa$ has the  $\Sigma_n(x,z)$-colouring property, then $\kappa$ has the $\Sigma_n(z)$-colouring property in $\LL[x]$. 
\end{corollary}

\begin{proof}
 Pick $A\in\POT{\kappa}^{\LL[x]}$  such that $\{A\}$ is $\Sigma_n(\kappa,z)$-definable in $\LL[x]$. 
Then $\{A\}$ is $\Sigma_n(\kappa,x,z)$-definable. 
 Let $\theta>\kappa$ be minimal with $z,A\in\LL_\theta[x]\models\ZFC^-$, 
let $b$ be the $<_{\LL[x]}$-least bijection between $\kappa$ and $\LL_\theta[x]$ in $\LL[x]$,
 let $\vartheta>\theta$ be minimal such that $b\in\LL_\theta[x]\models\ZFC^-$ and let $c$ be the $<_{\LL[x]}$-least bijection between $\kappa$ and $\LL_\vartheta[x]$ in $\LL[x]$. 
 Set $B = \Set{\goedel{\alpha}{\beta}}{\alpha,\beta<\kappa, ~ c(\alpha)\in c(\beta)} \in \POT{\kappa}^{\LL[x]}$. Then the set $\{B\}$ is $\Sigma_1(\kappa,x,z,A)$-definable and therefore it is also $\Sigma_n(\kappa,x,z)$-definable.  
By Lemma \ref{lemma:CharacterizationsPartitionProperty}, there is a weak $\kappa$-model $M$, a transitive set $N$ and an elementary embedding $\map{j}{M}{N}$ with $\crit{j}=\kappa$ and $B\in M$. 
Then $A\in\LL_\vartheta[x]\in M$ and $j(x)=x$. 
If we define $E=\Set{\langle\alpha,\beta\rangle\in\kappa\times\kappa}{b(\alpha)\in b(\beta)}$, then we know that $E\in\LL_\vartheta[x]\in M$ and $j(E)\in\LL_{j(\vartheta)}[x]\subseteq\LL[x]$. Moreover, since the function $j(b)$ is the transtive collapse of $\langle j(\kappa),j(E)\rangle$, we know that  $j(b)$ is also contained in $\LL[x]$. 
Finally, an easy computation  shows that $$\map{j\restriction\LL_\theta[x]=j(b)\circ b^{{-}1}}{\LL_\theta[x]}{\LL_{j(\theta)}[x]}$$ is an elementary embedding contained in $\LL[x]$. 
 By Lemma \ref{lemma:CharacterizationsPartitionProperty}, these  computations show that  $\kappa$ has the $\Sigma_n(z)$-colouring property in $\LL[x]$.  
\end{proof}


The concept introduced in the next definition will allow us to further  strengthen the above conclusions. Moreover, it will enable us to show that for all $n\geq 2$, the $\mathbf{\Sigma}_n$-colouring property is equivalent to the $\mathbf{\Sigma}_2$-colouring property.

\begin{definition}
 Given sets $z_0,\ldots,z_{m-1}$, a class $A$ \emph{has a good $\Sigma_n(z_0,\ldots,z_{m-1})$-well-ordering} if there is a well-ordering $\lhd$ of a class $B$ such that $A\subseteq B$ and the class $I(\lhd) = \Set{\Set{y}{y\lhd x}}{x\in B}$ of all proper initial segments of $\lhd$ is $\Sigma_n(z_0,\ldots,z_{m-1})$-definable. 
\end{definition}

It is easy to see that the canonical well-ordering of the constructible universe witnesses that the class $\LL$ has a good $\Sigma_1$-well-ordering. 
More generally, for every set of ordinals $x$, the class $\LL[x]$ has a good $\Sigma_1(x)$-well-ordering. 
Moreover, in the \emph{Dodd--Jensen core model $\KK^{DJ}$}, there is a good $\Sigma_1(\kappa)$-well-ordering of $\POT{\kappa}$ for every uncountable cardinal $\kappa$ (see {\cite[Lemma 1.10.]{GoodWO}}). 
Finally, it can also easily be shown that the canonical well-ordering of the collection $\HOD_x$ of all hereditarily $x$-ordinal-definable witnesses that $\HOD_x$ has a good $\Sigma_2(x)$-well-ordering (see  the proof of {\cite[Lemma 13.25]{MR1940513}} for details). 
Since the assumption $\VV=\HOD$ is compatible with various large cardinal assumptions and forcing axioms, this shows that the existence a good $\Sigma_2$-well-ordering of $\VV$ is also consistent with these extensions of $\ZFC$ and this will imply that it is most interesting to study the influence of these extensions of $\ZFC$ on the $\mathbf{\Sigma}_1$-colouring property. 
Moreover, this fact will allow us to show that the $\mathbf{\Sigma}_2$-colouring property implies all higher partition properties. This implication will be an easy consequence of the following observation.

\begin{proposition}\label{proposition:IntersectionsWithHOD}
 Let $X$ be a class of sets of ordinals with the property that both $X$ and $\VV\setminus X$ are $\Sigma_n(y)$-definable for some set $y$. If $z$ is a set with the property that $\HOD_z\cap X\neq\emptyset$, then there is an $A\in\HOD_z\cap X$ such that the set $\{A\}$ is $\Sigma_n(y,z)$-definable.  
\end{proposition}

\begin{proof}
 Let $\lhd$ denote the canonical well-ordering of $\HOD_z$ and let $A$ be the $\lhd$-least element in $\HOD_z\cap X$. By the above remarks, the class $I(\lhd)$ is definable by a $\Sigma_2$-formula with parameter $z$. Then $A$ is the unique element of $X$ with the property that there is a $D\in I(\lhd)$ with $D\cup\{A\}\in I(\lhd)$ and $D\cap X=\emptyset$. By our assumptions, this shows that the set $\{A\}$ is $\Sigma_n(y,z)$-definable.  
\end{proof}

The following corollary now shows that the validity of the $\mathbf{\Sigma}_2$-colouring property at a cardinal $\kappa$ is equivalent to the assumption that for every function $\map{c}{[\kappa]^2}{2}$ that is ordinal definable with parameters from $\HH{\kappa}$, there is a $c$-homogeneous set that is unbounded in $\kappa$.

\begin{corollary}\label{corollary:2ImpliesN}
 The following statements are equivalent for every uncountable regular cardinal $\kappa$ and every set $z$: 
  \begin{enumerate}[leftmargin=0.9cm]
   \item $\kappa$ has the $\Sigma_2(z)$-colouring property. 

   \item $\kappa$ has the $\Sigma_n(z)$-colouring property for all $n<\omega$. 

   \item For every function $\map{c}{[\kappa]^2}{2}$ in $\HOD_z$, there is a $c$-homogeneous set that is unbounded in $\kappa$. 
  \end{enumerate}
\end{corollary}

\begin{proof}
 Assume that there is a function $\map{c}{[\kappa]^2}{2}$ in $\HOD_z$ with the property that every $c$-homogeneous set is bounded in $\kappa$. Given $A\subseteq\kappa$, we let $\map{c_A}{[\kappa]^2}{2}$ denote the unique function with the property that for all $\alpha<\beta<\kappa$, we have $c(\alpha,\beta)=1$ if and only if $\goedel{\alpha}{\beta}\in A$. Let $X$ denote the set of all $A\subseteq\kappa$ with the property that every $c_A$-homogeneous set is bounded in $\kappa$. Then both $X$ and $\VV\setminus X$ are $\Sigma_2(\kappa)$-definable and our assumptions implies that $\HOD_z\cap X\neq\emptyset$. Therefore, we can apply Proposition \ref{proposition:IntersectionsWithHOD} to find an $A\in\HOD_z\cap X$ with the property that the set $\{A\}$ is $\Sigma_2(\kappa,z)$-definable. But this shows that $c_A$ is a $\Sigma_2(z)$-partition and therefore (i) fails. 
\end{proof}

Next, we show that the $\mathbf{\Sigma}_2$-colouring property is equivalent to weak compactness in certain canonical models of set theory.

\begin{proposition}\label{proposition:Sigma2PartitionPropertyWeaklyCompact}
 Let $\kappa$ be an uncountable regular cardinal with the property that there is a good $\Sigma_n(\kappa,z)$-well-ordering of $\POT{\kappa}$ for some set $z$. If $\kappa$ has the $\Sigma_2(z)$-colouring property, then $\kappa$ is weakly compact. 
\end{proposition}

\begin{proof}
 Assume that $\kappa$ is not weakly compact. 
 Let $\lhd$ be a well-ordering of a class $B$ such that $\POT{\kappa}\subseteq B$ and the corresponding class $I(\lhd)$  is $\Sigma_n(\kappa,z)$-definable. 
We define the colourings $\map{c_A}{[\kappa]^2}{2}$ for all $A\subseteq\kappa$ and the corresponding non-empty set $X\subseteq\POT{\kappa}$ as in the proof of Corollary \ref{corollary:2ImpliesN}. 
Let $B$ denote the $\lhd$-least element of $X$. Then $B\in X$ is the unique subset of $\kappa$ with the property that there exists a $D\in I(\lhd)$ with $D\cap X=\emptyset$ and $D\cup\{B\}\in I(\lhd)$. 
Since both $X$ and $\VV\setminus X$ are $\Sigma_2(\kappa)$-definable, our assumptions imply that the set $\{B\}$ is $\Sigma_{n+1}(\kappa,z)$-definable and hence $c_B$ is a $\Sigma_{n+1}(z)$-partition. 
Using Corollary \ref{corollary:2ImpliesN}, we can conclude that $\kappa$ does not have the $\Sigma_2(z)$-colouring property. 
\end{proof}

\begin{corollary}\label{corollary:LowerBoundsSigma2}
 Let $\kappa$ be an uncountable regular cardinal.
  \begin{enumerate}[leftmargin=0.9cm]
   \item If there is a set $z$ such that $\VV=\HOD_z$ and $\kappa$ has the $\Sigma_2(z)$-colouring property, then $\kappa$ is weakly compact. 

  \item If $\kappa$ has the $\Sigma_2(z)$-colouring property for some $z\in\HH{\kappa}\cap\POT{\kappa}$, then $\kappa$ is weakly compact in $\LL[z]$. \qed
 \end{enumerate} 
\end{corollary}

The above results show that there is a natural correspondence between the $\mathbf{\Sigma}_2$-colouring property and weak compactness. 
 In contrast, the results of this paper will show that the $\mathbf{\Sigma}_1$-colouring property corresponds to a large cardinal property that is weaker than weak compactness but stronger than Mahloness. 
In the following, we strengthen earlier results by showing that cardinals with the $\mathbf{\Sigma}_1$-colouring property possess a high degree of Mahloness in the constructible universe. 
An upper bound for the consistency strength of the $\mathbf{\Sigma}_1$-colouring property will be given by Theorem \ref{theorem:Sigma1WCbelowWC} below.

\begin{proposition}\label{proposition:FailureSigma1PPgoodWO}
 Let $\kappa$ be an uncountable regular cardinal and let $\nu<\kappa$ be a cardinal with $2^\nu\geq\kappa$. If there is a good $\Sigma_n(\kappa,z)$-well-ordering of $\HH{\kappa}$, then $\kappa$ does not have the $\Sigma_n(\nu,z)$-colouring property. 
\end{proposition}

\begin{proof}
 By our assumptions, we can use the good $\Sigma_n(\kappa,z)$-well-ordering of $\HH{\kappa}$ to construct a $\Sigma_n(\kappa,\nu,z)$-definable injection of $\kappa$ into ${}^\nu 2$. 
By Proposition \ref{Proposition:NoDefinableInjection}, the existence of such an injection contradicts the  $\Sigma_n(\nu,z)$-colouring property. 
\end{proof}


The following lemma generalizes the simultaneous reflection of stationary subsets of weakly compact cardinals to our definable setting. 
Note that the assumptions of the next lemma are satisfied in $\LL$ for every uncountable regular cardinal.

\begin{lemma}\label{lemma:DefinableSimultaneousReflection}
 Let $\kappa$ be an uncountable regular cardinal with the $\mathbf{\Sigma}_1$-colouring property. Assume that for some  $z\in\HH{\kappa}$, the set $\{\HH{\kappa}\}$ is $\Sigma_1(\kappa,z)$-definable and $\POT{\kappa}$ has a good $\Sigma_1(\kappa,z)$-well-ordering. If $\map{s}{\kappa}{\POT{\kappa}}$ is a $\Sigma_1(\kappa,z)$-definable function with the property that $s(\alpha)$ is stationary in $\kappa$ for all $\alpha<\kappa$, then the set $$S  =  \Set{\mu<\kappa}{\textit{$\mu$ is a regular cardinal with $s(\alpha)\cap\mu$  stationary in $\mu$ for all $\alpha<\mu$}}$$ is also stationary in $\kappa$. 
\end{lemma}

\begin{proof}
 Assume that the above conclusion fails. 
Note that the assumption that the set $\{\HH{\kappa}\}$ is $\Sigma_1(\kappa,z)$-definable implies that the set $\{S\}$ is definable in the same way. 
 Let $\lhd$ be a well-ordering of a class $B$ such that $\POT{\kappa}\subseteq B$ and the corresponding class $I(\lhd)$ is $\Sigma_1(\kappa,z)$-definable. 
 Let $C$ denote the $\lhd$-least club in $\kappa$ with $C\cap S=\emptyset$. 
 Then the set $\{C\}$ is also $\Sigma_1(\kappa,z)$-definable. 
 By Lemma \ref{lemma:CharacterizationsPartitionProperty},   there is a weak $\kappa$-model $M$, a transitive set $N$ and an elementary embedding $\map{j}{M}{N}$ with critical point $\kappa$ and $C,S\in M$. Since $\kappa\in j(C)$ and $\kappa$ is regular in $N$, elementarity yields an $\alpha<\kappa$ and a club subset $D$ of $\kappa$ in $N$ with $D\cap j(s)(\alpha)=\emptyset$. But $j(s)(\alpha)\cap\kappa=j(s(\alpha))\cap\kappa=s(\alpha)$ and hence $D$ witnesses that $s(\alpha)$ is not stationary in $\kappa$, a contradiction. 
\end{proof}

  Remember that, given an inaccessible cardinal $\kappa$ and an ordinal $\delta\leq\kappa^+$, the cardinal $\kappa$ is \emph{$\delta$-Mahlo} if there is a sequence $\seq{A_\gamma}{\gamma<\delta}$ of stationary subsets of $\kappa$ such that the following statements hold for all $\gamma<\delta$: 
  \begin{enumerate}[leftmargin=0.9cm]
   \item $A_0=\Set{\alpha<\kappa}{\textit{$\alpha$ is regular}}$. 
   
   \item If $\gamma=\beta+1$, then $A_\gamma=\Set{\alpha\in A_\beta}{\textit{$A_\beta\cap\alpha$ is stationary in $\alpha$}}$. 
   
   \item If $\gamma$ is a limit ordinal of cofinality less than $\kappa$, then there is a strictly increasing sequence $\seq{\beta_\alpha}{\alpha<\cof{\gamma}}$ that is cofinal in $\gamma$ with $A_\gamma=\bigcap\Set{A_{\beta_\alpha}}{\alpha<\cof{\gamma}}$. 
   
   \item If $\gamma$ is a limit ordinal of cofinality $\kappa$, then there is a strictly increasing sequence $\seq{\beta_\alpha}{\xi<\kappa}$ that is cofinal in $\gamma$ with $A_\gamma=\bigtriangleup\Set{A_{\beta_\alpha}}{\alpha<\kappa}$. 
  \end{enumerate}
      
A cardinal $\kappa$ is then called \emph{hyper-Mahlo} if it is $\kappa$-Mahlo. 
 Note that, given two sequences $\seq{A_\beta}{\beta<\alpha}$ and $\seq{B_\beta}{\beta<\alpha}$ of subsets of $\kappa$ that satisfy the above four statements and some $\beta<\alpha$, the sets $A_\beta$ and $B_\beta$ only differ by a non-stationary subset of $\kappa$. In particular, if $\kappa$ is an inaccessible cardinal that is  not $\delta$-Mahlo for some  $\delta\leq\kappa^+$, then there is a $\gamma<\delta$, such that $\kappa$ is $\gamma$-Mahlo and not $(\gamma+1)$-Mahlo.

\begin{theorem}\label{theorem:DegreesOfMahlonessFromPP}
 Let $\kappa$ be an uncountable regular cardinal with the $\mathbf{\Sigma}_1$-colouring property. Assume that for some  $z\in\HH{\kappa}$, the set $\{\HH{\kappa}\}$ is $\Sigma_1(\kappa,z)$-definable and $\POT{\kappa}$ has a good $\Sigma_1(\kappa,z)$-well-ordering. Define $\sigma$ to be the supremum of all ordinals $\delta$ with the property that there is a subset $E$ of $\kappa\times\kappa$ such that $\langle\kappa,E\rangle$ is a well-ordering of order-type $\delta$ and the set $\{E\}$ is $\Sigma_1(\kappa,w)$-definable for some $w\in\HH{\kappa}$. Then $\kappa$ is a $\sigma$-Mahlo cardinal. 
\end{theorem}

\begin{proof}
 Let $\lhd$ be a well-ordering of a class $D$ such that $\POT{\kappa}\subseteq D$ such that the  class $I(\lhd)$ is $\Sigma_1(\kappa,z)$-definable. 
 Assume that the above conclusion fails. 
 Since Proposition \ref{proposition:FailureSigma1PPgoodWO} shows that $\kappa$ is inaccessible,  the above remarks show that there is a $\delta<\sigma$ such that $\kappa$ is $\delta$-Mahlo and not $(\delta+1)$-Mahlo. 
Pick $E\subseteq\kappa\times\kappa$ such that $\langle\kappa,E\rangle$ is a well-ordering of order-type at least $\delta$ and the set $\{E\}$ is $\Sigma_1(\kappa,w)$-definable for some $w\in\HH{\kappa}$. 
Then we can find $\lambda\leq\kappa$ and $y\in\HH{\kappa}$ such that the set $\{\delta\}$ is $\Sigma_1(\kappa,y)$-definable and there is a $\Sigma_1(\kappa,y)$-definable bijection $\map{b}{\lambda}{\delta}$. 
 %
 %
Let $\seq{A_\gamma}{\gamma\leq\delta}$ denote the unique sequence of subsets of $\kappa$ such that for all $\gamma\leq\delta$,  the above statements (i) and (ii) as well as the following two statements hold: 
 \begin{enumerate}[leftmargin=0.9cm]
  \item[(iii)${}^\prime$]  If $\gamma\in\Lim$ with $\cof{\gamma}<\kappa$ and $c_\gamma$ is the $\lhd$-least subset of $\lambda$ of cardinality less than $\kappa$ with the property that $b[c_\gamma]$ is a cofinal subset of $\gamma$ of order-type $\cof{\gamma}$, then $A_\gamma=\bigcap\Set{A_{b(\beta)}}{\beta\in c_\gamma}$.

  \item[(iv)${}^\prime$]  If $\gamma\in\Lim$ with $\cof{\gamma}=\kappa$ and $c_\gamma$ is the $\lhd$-least subset of $\kappa$ such that $b\restriction b_\gamma$ is strictly increasing and $b[c_\gamma]$ is a cofinal subset of $\gamma$ of order-type $\kappa$, then  $A_\gamma=\bigtriangleup\Set{A_{b(\beta^\gamma_\alpha)}}{\alpha<\kappa}$, where $\seq{\beta^\gamma_\alpha}{\alpha<\kappa}$ denotes the monotone enumeration of $c_\gamma$. 
 \end{enumerate}

Then the sequence $\seq{A_\gamma}{\gamma\leq\delta}$ also satisfies the above properties (iii) and (iv). 
 Therefore, our assumptions imply that $A_\gamma$ is a stationary subset of $\kappa$ for every $\gamma<\delta$ and that  there is a club $D$ in $\kappa$ that is disjoint from $A_\delta$.
Moreover, by combining the $\Sigma_1$-Recursion Theorem with the fact that the set $\{\HH{\kappa}\}$ and the function $b$ are both $\Sigma_1(\kappa,y,z)$-definable, we know that the set $\{\seq{A_\gamma}{\gamma\leq\delta}\}$ is $\Sigma_1(\kappa,y,z)$-definable and therefore the function $$\Map{s}{\lambda}{\POT{\kappa}}{\beta}{A_{b(\beta)}}$$ is definable in the same way.  
If we now let $S$ denote the set of all regular cardinals $\mu<\kappa$ with the property that $s(\alpha)\cap\mu$ is stationary in $\mu$ for all $\alpha<\min\{\lambda,\mu\}$, then Lemma \ref{lemma:DefinableSimultaneousReflection} shows that $S$ is stationary in $\kappa$.

\begin{claim*}
 For all $\gamma\leq\delta$, there is a club $C_\gamma$ in $\kappa$ with $C_\gamma\cap S\subseteq A_\gamma$. 
\end{claim*}

\begin{proof}[Proof of the Claim]
 We prove the claim by induction on $\gamma\leq\delta$. 
First, since $S\subseteq A_0$, we can define $C_0=\kappa$. 
Now, if $\gamma<\delta$ and $C_\gamma$ is already constructed, then we define  $C_{\gamma+1}=C_\gamma\cap(b^{{-}1}(\gamma),\kappa)$. 
Given $\mu\in C_{\gamma+1}\cap S$, we then have $\mu\in A_\gamma$, $b^{{-}1}(\gamma)<\min\{\lambda,\mu\}$ and  hence  $A_\gamma\cap\alpha=s(b^{{-}1}(\gamma))\cap\mu$ is stationary in $\alpha$. 
This shows that $C_{\gamma+1}\cap S\subseteq A_{\gamma+1}$.
  Next, if $\gamma\in\Lim\cap(\delta+1)$ with $\cof{\gamma}<\kappa$ and $C_\beta$ is defined for all $\beta<\gamma$, then we define $C_\gamma=\bigcap\Set{C_{b(\beta)}}{\beta\in c_\gamma}$. 
Then the definition of $A_\gamma$ directly implies that $C_\gamma\cap S\subseteq A_\gamma$. 
Finally, assume that $\gamma\in\Lim\cap(\delta+1)$ with $\cof{\gamma}=\kappa$ and $C_\beta$ is defined for all $\beta<\gamma$. 
Set $C_\gamma=\bigtriangleup\Set{C_{b(\beta^\gamma_\alpha)}}{\alpha<\kappa}$. 
Given $\mu\in C_\gamma\cap S$, we then have $\mu\in C_{b(\beta^\gamma_\alpha)}\cap S\subseteq A_{b(\beta^\gamma_\alpha)}$ for all $\alpha<\mu$ and this allows us to conlcude that $\mu\in A_\gamma=\bigtriangleup\Set{A_{b(\beta^\gamma_\alpha)}}{\alpha<\kappa}$. 
\end{proof}

But, now we have $\emptyset\neq C_\delta\cap D\cap S\subseteq A_\delta\cap D$, a contradiction.  
\end{proof}

Note that the assumptions of Theorem \ref{theorem:DegreesOfMahlonessFromPP} are satisfied if $\kappa$ has the $\mathbf{\Sigma}_1$-colouring property and there is an $A\subseteq\kappa$ such that $\POT{\kappa}\subseteq\LL[A]$ and the set is $\{A\}$ is $\Sigma_1(\kappa,z)$-definable for some $z\in\HH{\kappa}$. 
Moreover, Proposition \ref{proposition:SmallForcingDef} below  shows that we can force over $\LL$ to show that the conclusion of the above theorem can fail if one discards the assumption that there is a good $\Sigma_1$-well-ordering of $\POT{\kappa}$.

\begin{corollary}
 Let $\kappa$ be an uncountable regular cardinal with the $\mathbf{\Sigma}_1$-colouring property. 
 \begin{enumerate}[leftmargin=0.9cm]
  \item If there is $A\subseteq\kappa$ such that $\VV=\LL[A]$ and the set $\{A\}$ is $\Sigma_1(\kappa,z)$-definable for some $z\in\HH{\kappa}$, then $\kappa$ is a hyper-Mahlo cardinal. 

 \item If $x\in\HH{\kappa}\cap\POT{\kappa}$, then $\kappa$ is a hyper-Mahlo cardinal in $\LL[x]$. \qed
 \end{enumerate}
\end{corollary}

In the remainder of this section, we will show that the validity of the $\mathbf{\Sigma}_n$-colouring property at the successor of a regular cardinal is equiconsistent with both the existence of an inaccessible cardinal with the $\mathbf{\Sigma}_n$-colouring  property and the existence of an accessible limit cardinals with this property. 
By the above results, this will show that, in the case $n=2$, all of these corresponding theories are equiconsistent to the existence of a  weakly compact cardinal. 
In the case $n=1$, the above computations and Theorem \ref{theorem:Sigma1WCbelowWC} below will show that the consistency strength of the given theories lies strictly between the existence of a hyper-Mahlo cardinal and a weakly compact cardinal. 
Finally, our results will also show that the $\mathbf{\Sigma}_2$-colouring property does not imply Mahloness for inaccessible cardinals.

\begin{lemma}\label{lemma:LevySolovayPartitionProperty}
 Let $\kappa$ be an uncountable regular cardinal with the $\mathbf{\Sigma}_n$-colouring property and let $\PPP\in\HH{\kappa}$ be a partial order. If $G$ is $\PPP$-generic over $\VV$, then $\kappa$ has the $\mathbf{\Sigma}_n$-colouring property in $\VV[G]$. 
\end{lemma}

\begin{proof}
 Fix $z\in\HH{\kappa}^{\VV[G]}$, $A\in\POT{\kappa}^{\VV[G]}$ and a $\Sigma_n$-formula $\varphi(v_0,v_1,v_2)$ such that $A$ is the unique set in $\VV[G]$ with the property that $\varphi(\kappa,z,A)$ holds in $\VV[G]$. 
 Since $\PPP$ is an element of $\HH{\kappa}^\VV$, there is a $\PPP$-name $\dot{z}$ in $\HH{\kappa}^\VV$ with $z=\dot{z}^G$. 
Pick a condition $p$ in $G$ that forces the above statements about $\dot{z}$ to hold true and fix a bijection $b$ between a cardinal $\nu<\kappa$ and the set of all conditions below $p$ in $\PPP$. Set $$B ~ = ~ \Set{\goedel{\alpha}{\gamma}}{\alpha<\kappa, ~ \gamma<\nu, ~ b(\gamma)\Vdash^\VV_\PPP\anf{\exists x ~ [\varphi(\check{\kappa},\dot{z},x) ~ \wedge ~ \check{\alpha}\in x]}} ~ \in ~ \POT{\kappa}^\VV.$$

 A careful review of the definition of the forcing relation (see, for example, {\cite[Section VII.3]{MR597342}}) shows that for every  $\Sigma_n$-formula $\psi(v_0,\ldots,v_{m-1})$, there is a $\Sigma_n$-formula $\psi(v_0,\ldots,v_{m+1})$ such that the axioms of $\ZFC^-$ prove that for every partial order $\PPP$, every  $p$ in $\PPP$ and all $\tau_0,\ldots,\tau_{n-1}$, the statement $\psi(\tau_0,\ldots,\tau_{n-1},\PPP,p)$ holds if and only if the sets $\tau_0,\ldots,\tau_{n-1}$ are $\PPP$-names with $p\Vdash_\PPP\varphi(\tau_0,\ldots,\tau_{n-1})$. 
In particular, the set $B$ is $\Sigma_n(\kappa,\nu,b,\dot{z},\PPP)$-definable in $\VV$. 
Moreover, if $\alpha<\kappa$ and $\gamma<\nu$, then  $\goedel{\alpha}{\gamma}$ is not contained in $B$ if and only if there is a $\delta<\nu$ with $b(\delta)\leq_\PPP b(\gamma)$ and $b(\delta)  \Vdash^\VV_\PPP  \anf{\exists x ~ [\varphi(\check{\kappa},\dot{z},x)  \wedge  \check{\alpha}\notin x]}$.  
This shows that the set $\kappa\setminus B$ is also $\Sigma_n(\kappa,\nu,b,\dot{z},\PPP)$-definable in $\VV$ and this allows us to conclude that the set $\{B\}$ is definable in the same way. 
By our assumption and Lemma \ref{lemma:CharacterizationsPartitionProperty},  there is an elementary embedding $\map{j}{M}{N}$ with critical point $\kappa$ in $\VV$ such that  $M$ is a weak $\kappa$-model, $N$ is  transitive, $\dot{z},B,\PPP\in M$ and $\kappa$ is inaccessible in $M$. 
Since $j\restriction\PPP=\id_\PPP$, this embedding has a canonical lift $\map{j_G}{M[G]}{N[G]}$ in $\VV[G]$ (see {\cite[Proposition 9.1]{MR2768691}}). 
But then $A$ consists of all $\alpha<\kappa$ with the property that there is a $\gamma<\nu$ with $b(\gamma)\in G$ and $\goedel{\alpha}{\gamma}\in B$. 
This shows that $A$ is an element of $M[G]$. 
Since $\kappa$ is inaccessible in $M[G]$, Lemma \ref{lemma:CharacterizationsPartitionProperty} shows that $\kappa$ has the $\mathbf{\Sigma}_n$-colouring property in $\VV[G]$. 
\end{proof}

\begin{proposition}\label{proposition:SmallForcingDef}
 Let $\kappa$ be an uncountable regular cardinal. 
 \begin{enumerate}[leftmargin=0.9cm]
 \item Let $\mu<\kappa$ be an infinite regular cardinal and let $\PPP\in\{\Add{\mu}{\kappa},\Col{\mu}{{<}\kappa}\}$. If either $\mu=\omega$ or $\VV=\LL$ holds, then the set $\{\PPP\}$ is $\Sigma_1(\kappa,\mu)$-definable. 

  \item Let $\PPP$ be a weakly homogeneous partial order with the property that the set $\{\PPP\}$ is $\Sigma_n(\kappa,y)$-definable for some set $y$. If $G$ is $\PPP$-generic over $\VV$ and $A$ is a subset of $\kappa$ in $\VV[G]$ with the property that the set $\{A\}$ is $\Sigma_n(\kappa,z)$-definable in $\VV[G]$ for some $z\in\VV$, then $A$ is an element of $\VV$ and the set $\{A\}$ is $\Sigma_n(\kappa,y,z)$-definable in $\VV$. 

 \item If $\kappa$ has the $\mathbf{\Sigma}_2$-colouring  property, then in a generic extension $\VV[G]$ of $\VV$ with $\HH{\kappa}^{\VV[G]}\subseteq\VV$, the cardinal $\kappa$ has the $\mathbf{\Sigma}_2$-colouring property and is not Mahlo. 
 \end{enumerate}
\end{proposition}

\begin{proof}
 (i) Our assumptions imply $\Add{\mu}{\kappa}=\Add{\mu}{\kappa}^\LL$ and $\Col{\mu}{{<}\kappa}=\Col{\mu}{{<}\kappa}^\LL$ and this shows that $\PPP$ is definable over $\langle\LL_\kappa,\in\rangle$ by a formula with parameter $\mu$. Since the set $\{\LL_\kappa\}$ is $\Sigma_1(\kappa)$-definable, we know that the set $\{\PPP\}$ is $\Sigma_1(\kappa,\mu)$-definable.

 (ii) Pick a $\Sigma_n$-formula $\varphi(v_0,v_1,v_2)$ such that $A$ is the unique set in $\VV[G]$ with the property that $\varphi(A,\kappa,z)$ holds. Then the weak homogeneity of $\PPP$ in $\VV$ implies $$A ~ = ~ \Set{\alpha<\kappa}{\mathbbm{1}_\PPP\Vdash^\VV\anf{\exists X ~ [\check{\alpha}\in X ~ \wedge ~ \varphi(X,\check{\kappa},\check{z})}]} ~ \in ~ \VV$$ and, by the remarks made in the proof of Lemma \ref{lemma:LevySolovayPartitionProperty}, this shows that the set $\{A\}$ is $\Sigma_n(\kappa,z,\PPP)$-definable in $\VV$. By our assumptions on $\PPP$, this shows that $\{A\}$ is $\Sigma_n(\kappa,y,z)$-definable in $\VV$. 

 (iii) Let $S$ denote the set of all singular limit ordinals less than $\kappa$. Then $S$ is a fat stationary subset of $\kappa$ and the  canonical partial order $\CCC(S)$  that shoots a club through $S$ using bounded closed subsets of $S$ is ${<}\kappa$-distributive (see {\cite[Section 1]{MR716625}}). Moreover, the set $\{\CCC(S)\}$ is $\Sigma_2(\kappa)$-definable and  {\cite[Section 3.5, Theorem 1]{MR0373889}} implies that $\CCC(S)$ is weakly homogeneous. 
Let $G$ be $\CCC(S)$-generic over $\VV$ and let $A$ be a subset of $\kappa$ in $\VV[G]$ such that the set $\{A\}$ is $\Sigma_2(\kappa,z)$-definable for some $z\in\HH{\kappa}^{\VV[G]}$. By (ii) and the above remarks, we know that $A,z\in\VV$ and the set $\{A\}$ is $\Sigma_2(\kappa,z)$-definable in $\VV$. Hence our assumptions allow us to use Lemma \ref{lemma:CharacterizationsPartitionProperty} to find a weak $\kappa$-model $M$, a transitive set $N$ and an elementary embedding $\map{j}{M}{N}$ in $\VV$ such that $A\in M$, $\crit{j}=\kappa$, $\kappa$ inaccessible in $M$ and $\HH{\kappa}^M\in M$. 
Since these properties of $M$, $N$ and $j$ are upwards absolute to $\VV[G]$, Lemma \ref{lemma:CharacterizationsPartitionProperty} shows that $\kappa$ has the $\mathbf{\Sigma}_n$-colouring property in $\VV[G]$.  
\end{proof}

\begin{lemma}\label{lemma:Sigma1PartitionPropertyCollapses}
 Let $\kappa$ be an inaccessible cardinal with the $\mathbf{\Sigma}_n$-colouring property, let $\mu<\kappa$ be an infinite regular cardinal, let $\PPP\in\{\Add{\mu}{\kappa},\Col{\mu}{{<}\kappa}\}$ and let $G$ be $\PPP$-generic over $\VV$. If either $\mu=\omega$, or $\VV=\LL$ holds, or $\kappa$ is weakly compact in $\VV$, then $\kappa$ has the $\mathbf{\Sigma}_n$-colouring property in $\VV[G]$.  
\end{lemma}

\begin{proof}
 Fix a $A\in\POT{\kappa}^{\VV[G]}$ that is $\Sigma_n(\kappa,z)$-definable in $\VV[G]$ for some $z\in\HH{\kappa}^{\VV[G]}$. 
  Then there is a regular cardinal $\mu<\nu<\kappa$ and $H\in\VV[G]$ such that $z\in\VV[H]$, $H$ is either a $\Add{\mu}{\nu}$- or $\Col{\mu}{\nu}$-generic over $\VV$ and $\VV[G]$ is a $\PPP$-generic extension of $\VV[H]$.  
Moreover, Lemma \ref{lemma:LevySolovayPartitionProperty} implies that $\kappa$ has the $\mathbf{\Sigma}_n$-colouring property in $\VV[H]$. 
 By our assumptions, Proposition \ref{proposition:SmallForcingDef} shows that there is a $y\in\VV$ such that,  in $\VV[H]$, the set $\{\PPP\}$ is $\Sigma_1(\kappa,y)$-definable and $\kappa$ has the $\Sigma_n(\kappa,y,z)$-colouring property. 
 Another application of Proposition \ref{proposition:SmallForcingDef} shows that $A$ is an element of $\VV[H]$ and the set $\{A\}$ is $\Sigma_n(\kappa,\mu,z)$-definable in $\VV[H]$. 
As in the last part of the proof of Proposition \ref{proposition:SmallForcingDef}, these computations show that $\kappa$ has the $\mathbf{\Sigma}_n$-colouring property in $\VV[G]$. 
\end{proof}


\section{The $\Sigma_n$-club property}\label{section:SigmaNclubproperty}

In this section, we will provide an analysis of the $\Sigma_n$-club property that parallels the investigation of the $\Sigma_n$-colouring  property in the last section. 
In particular, we will show that $\omega_1$ is the only uncountable cardinal that can consistently possess the $\Sigma_2$-club property and the only successor cardinal that can consistently have the $\mathbf{\Sigma}_1$-club property. 
In contrast, the results of this paper will show that $\omega_1$ can consistently possess the $\mathbf{\Sigma}_n$-club property for all $n<\omega$, several well-known large cardinal properties imply the $\mathbf{\Sigma}_1$-club property and the existence of an accessible limit cardinal with the $\mathbf{\Sigma}_1$-club property is consistent. 
Moreover, we will show that the $\mathbf{\Sigma}_2$-club property implies all higher club properties. 
Finally, we will again establish a natural connection between these properties and large cardinal properties. Our results will show that the $\mathbf{\Sigma}_2$-club property is naturally connected with measurability through the inner model $\HOD$ and that it is possible to use the \emph{Dodd--Jensen core model $\KK^{DJ}$} to connect the $\mathbf{\Sigma}_1$-club property to a large cardinal property that implies the existence of sharps for reals and is a consequence of \emph{$\omega_1$-iterability} (see Definition \ref{definition:IterableCardinal}). 
The following characterizations of the $\mathbf{\Sigma}_n$-club properties is the starting point of our analysis. 

\begin{lemma}\label{lemma:CharacterizationsClubProperty}
 The following statements are equivalent for every uncountable regular cardinal $\kappa$ and every set $z$: 
 \begin{enumerate}[leftmargin=0.9cm]
  \item Given $\gamma_0,\ldots,\gamma_{m-1}<\kappa$ and $A\subseteq\kappa$, if the set $\{A\}$ is $\Sigma_n(\kappa,\gamma_0,\ldots,\gamma_{m-1},z)$-definable, then either $A$ contains a club subset of $\kappa$ or is disjoint from such a subset. 
  
  \item Given $\gamma_0,\ldots,\gamma_m<\kappa$, every $\Sigma_n(\kappa,\gamma_0,\ldots,\gamma_{m-1},z)$-definable function $\map{c}{\kappa}{\gamma_m}$ is constant on a club subset of $\kappa$. 
  
  \item Given $\gamma_0,\ldots,\gamma_m<\kappa$, if $\map{c}{[\kappa]^k}{\gamma_m}$ is a $\Sigma_n(\kappa,\gamma_0,\ldots,\gamma_{m-1},z)$-partition, then there is a $c$-homogeneous club subset of $\kappa$. 
  
    \item Given $\gamma_0,\ldots,\gamma_m<\kappa$, if $\map{c}{[\kappa]^{{<}\omega}}{\gamma_m}$ is a $\Sigma_n(\kappa,\gamma_0,\ldots,\gamma_{m-1},z)$-definable function, then there is a $c$-homogeneous closed unbounded subset of $\kappa$. 
 \end{enumerate}
\end{lemma}

\begin{proof}
 First, assume that (i) holds. Fix a $\Sigma_n(\kappa,\gamma_0,\ldots,\gamma_{m-1},z)$-definable function $\map{c}{\kappa}{\gamma_m}$ with  $\gamma_0,\ldots,\gamma_m<\kappa$. 
 Given $\xi<\gamma_m$, define $A_\xi=\Set{\alpha<\kappa}{c(\alpha)=\xi}$. 
For every $\xi<\gamma_m$, the set $\{A_\xi\}$ is $\Sigma_n(\kappa,\gamma_0,\ldots,\gamma_{m-1},\xi,z)$-definable and therefore our assumption yields a club subset $C_\xi$ that is either contained in this set or disjoint from it. 
If $\alpha\in\bigcap\Set{C_\xi}{\xi<\gamma_m}$, then $\alpha\in A_{c(\alpha)}\cap C_{c(\alpha)}$, $C_{c(\alpha)}\subseteq A_{c(\alpha)}$ and therefore $c$ is constant on $C_{c(\alpha)}$ with value $c(\alpha)$.

 Now, assume  (ii)  and fix a $\Sigma_n(\kappa,\gamma_0,\ldots,\gamma_{m-1},z)$-partition $\map{c}{[\kappa]^k}{\gamma_m}$ with $\gamma_0,\ldots,\gamma_m<\kappa$. 
 In this situation, we can use (ii) to inductively construct  
 \begin{itemize}[leftmargin=0.9cm]
  \item  a sequence $\seq{\map{c^l_a}{(\max(a),\kappa)}{\gamma_m}}{l<k, ~ a\in[\kappa]^l}$ of functions, 
  
  \item a sequence $\seq{\xi^l_a<\gamma_m}{l<k, ~ a\in[\kappa]^l}$ of ordinals, and 
  
  \item a sequence  $\seq{\varphi_l(v_0,\ldots,v_{m+l+3})}{l<k}$ of $\Sigma_n$-formulas
 \end{itemize}
 such that the following statements hold:
  \begin{enumerate}[leftmargin=0.9cm]
    \item[(a)] If $a\in[\kappa]^{k-1}$ and $\max(a)<\alpha<\kappa$, then $c^{k-1}_a(\alpha) =  c(a\cup\{\alpha\})$.
    
    \item[(b)] If $l<k$ and $\alpha_0<\ldots<\alpha_{l-1}<\alpha<\kappa$, then $c^l_{\{\alpha_0,\ldots,\alpha_{l-1}\}}(\alpha)$ is the unique ordinal $\beta$ such that $\varphi_l(\kappa,\alpha_0,\ldots,\alpha_{l-1},\alpha,\beta,\gamma_0,\ldots,\gamma_{m-1},z)$ holds. 
    
    \item[(c)] If $l<k$ and $a\in[\kappa]^l$, then $\xi^l_a$ is the unique element of $\gamma_m$ whose preimage under $c^l_a$ contains a closed unbounded subset of $\kappa$. 
    
    \item[(d)] If $0<l<k$, $a\in[\kappa]^{l-1}$ and $\max(a)<\alpha<\kappa$, then $c^{l-1}_a(\alpha)=\xi^l_{a\cup\{\alpha\}}$. 
  \end{enumerate}

 Given $l<k$ and $a\in[\kappa]^l$, pick a club $C_a$ in $\kappa$ with $c^l_a[C_a]=\{\xi^l_a\}$. 
Define $$C ~ = ~ \bigtriangleup\Set{\bigcap\Set{C^l_a}{l<k, ~ a\in[\alpha]^l}}{\alpha<\kappa}.$$ 
Pick $\alpha_0,\ldots,\alpha_{k-1}\in C\cap\Lim$ with $\alpha_0<\ldots<\alpha_{k-1}$ and set $a_l=\{\alpha_0,\ldots,\alpha_{l-1}\}$ for all $l\leq k$. 
Then $\alpha_l\in C^l_{a_l}$ and $c^l_{a_l}(\alpha_l)=\xi^l_{a_l}$ for all $l<k$. 
 Moreover, if $0<l\leq k$, then $\xi^l_{a_l}=c^{l-1}_{a_{l-1}}(\alpha_{l-1})$. 
 In combination, this allows us to conclude that $c(a_k)=c^{k-1}_{a_{k-1}}(\alpha_{k-1})=\xi^0_\emptyset$ and this shows that $C$ is $c$-homogeneous.

 Next, assume  (iii)  and pick a $\Sigma_n(\kappa,\gamma_0,\ldots,\gamma_{m-1},z)$-definable function $\map{c}{[\kappa]^{{<}\omega}}{\gamma_m}$ with $\gamma_0,\ldots,\gamma_m<\kappa$.  Given $0<k<\omega$, the function $c\restriction[\kappa]^k$ is definable in the same way and our assumption yields a $(c\restriction[\kappa]^k)$-homogeneous club $C_k$ in $\kappa$. 
Then the club $\bigcap\Set{C_k}{0<k<\omega}$ is $c$-homogeneous.

 Finally, assume that (iv) holds. Pick $\gamma_0,\ldots,\gamma_{m-1},\gamma<\kappa$ and $A\subseteq\kappa$ such that the set $\{A\}$ is $\Sigma_n(\kappa,\gamma_0,\ldots,\gamma_{m-1},z)$-definable. 
 Let $\map{c}{[\kappa]^{{<}\omega}}{2}$ denote the unique function with the property that for all $a\in[\kappa]^{{<}\omega}$, we have $$c(a)=1 ~ \Longleftrightarrow ~ (a\neq\emptyset ~ \wedge ~ \min(a)\in A).$$  
 Then the function $c$ is $\Sigma_n(\kappa,\gamma_0,\ldots,\gamma_{m-1},z)$-definable and (iv) yields a $c$-homogeneous club $C$ in $\kappa$. We can conclude that either $C\subseteq A$ or $A\cap C=\emptyset$. 
\end{proof}

The above lemma now allows us to prove the restrictions on the possible types of cardinals possessing the club property that were mentioned above. 
Remember that, given regular cardinals $\mu<\nu$,  
we let $S^\nu_\mu$ denote the set of all limit ordinals $\lambda<\nu$ with $\cof{\lambda}=\mu$.

\begin{proposition}\label{proposition:FailuresClubProperty}
 \begin{enumerate}[leftmargin=0.9cm]
 \item The set $\{S^{\nu^+}_\nu\}$ is $\Sigma_1(\nu^+,\nu)$-definable for all infinite regular cardinals $\nu$. 

  \item The set $\{S^{\omega_{k+1}}_{\omega_k}\}$ is $\Sigma_1(\omega_{k+1})$-definable for all $k<\omega$. 

  \item If $\nu$ is an uncountable cardinal, then $\nu^+$ does not have the $\mathbf{\Sigma}_1$-club property.  

  \item Regular cardinals greater than $\omega_1$ do not have the $\Sigma_2$-club property.  
 \end{enumerate}
\end{proposition}

\begin{proof}
  (i) Fix an infinite regular cardinal $\nu$ and $\gamma\in\Lim\cap\nu^+$. 
 If there is a strictly increasing cofinal function $\map{s}{\nu}{\gamma}$, then $\cof{\gamma}=\nu$. 
 In the other case, if there is a limit ordinal $\lambda<\nu$ and a strictly increasing cofinal function $\map{s}{\lambda}{\gamma}$, then  $\cof{\gamma}<\nu$. 
 These two implications yield a $\Sigma_1(\nu^+,\nu)$-definition of $\{S^{\nu^+}_\nu\}$.

 (ii) Given $k<\omega$, the cardinal $\omega_k$ is the unique ordinal $\lambda$ with the property that there is a transitive model $M$ of $\ZFC+\anf{\textit{$\omega_\omega$ exists}}$ such that $\omega_{k+1}=\omega_{k+1}^M$ and $\omega_k=\lambda$. 
This observation shows that the set $\{\omega_k\}$ is  $\Sigma_1(\omega_{k+1})$-definable and, in combination with (i), this yields the desired statement.

 (iii) Assume, towards a contradiction, that there is an uncountable cardinal $\nu$ such that the cardinal $\nu^+$ has the $\mathbf{\Sigma}_1$-club property. 
 Since $\nu$ is uncountable, Lemma \ref{lemma:CharacterizationsClubProperty} and (i) imply that  $\nu$ is singular. 
 Let $z$ denote the set of all uncountable regular cardinals less than $\nu$. 
 Then the set  $S^{\nu^+}_\omega$ consists of all limit ordinals $\lambda<\nu^+$ with the property that there is no strictly increasing cofinal function $\map{s}{\mu}{\gamma}$ with $\mu\in z$. 
 This shows that $\{S^{\nu^+}_\omega\}$ is $\Sigma_1(\nu^+,z)$-definable, contradicting Lemma \ref{lemma:CharacterizationsClubProperty}.

 (iv) If $\kappa<\omega_1$ is a regular cardinal, then $\{S^\kappa_\omega\}$ is $\Sigma_2(\kappa)$-definable and therefore Lemma \ref{lemma:CharacterizationsClubProperty} implies that $\kappa$ does not have the $\Sigma_2$-club property. 
\end{proof}

Note that in general, if $\nu$ is an infinite cardinal, then the set $\{\nu\}$ need not be $\Sigma_1(\nu^+)$-definable. For example, {\cite[Corollary 3.3]{GoodWO}} shows that it is consistent that for some measurable cardinal $\delta$, the sets $\{\delta\}$ and $\{\delta^+\}$ are not $\Sigma_1(\delta^{++})$-definable.

\begin{proposition}
The following statements are equivalent for every uncountable regular cardinal $\kappa$: 
  \begin{enumerate}[leftmargin=0.9cm]
   \item $\kappa$ has the $\mathbf{\Sigma}_2$-club property. 

   \item $\kappa$ has the $\mathbf{\Sigma}_n$-club property for all $n<\omega$. 

   \item If $z\in\HH{\kappa}$ , then $\HOD_z$ does not contain a bistationary subset of $\kappa$. 
  \end{enumerate}
\end{proposition}

\begin{proof}
 Assume that (iii) fails for some $z\in\HH{\kappa}$ and let $X$ denote the set of all bistationary subsets of $\kappa$. Then both $X$ and $\VV\setminus X$ are $\Sigma_2(\kappa)$-definable and hence Proposition \ref{proposition:IntersectionsWithHOD} yields an $A\in X$ with the property that the set $\{A\}$ is $\Sigma_2(\kappa,z)$-definable. By Lemma \ref{lemma:CharacterizationsClubProperty}, this implies that (i) fails. 
\end{proof}

In the remainder of this section, we investigate the consistency strength of the $\mathbf{\Sigma}_n$-club properties. 
In the case $n=1$, we will show that the validity of the $\mathbf{\Sigma}_1$-club property at $\omega_1$ is equiconsistent with both the existence of an inaccessible cardinal with this property and the existence of an accessible limit cardinal possessing this property. Moreover, we will present narrow bounds for the consistency strength of these theories.  
In the following, many arguments rely on the notion of \emph{good sets of indiscernibles} (see {\cite[Section 1]{MR645907}}). 
Remember that, if $\kappa$ is a cardinal and $A$ is a subset of $\kappa$, then $I\subseteq\kappa$ is a \emph{good set of indiscernibles for $\langle\LL_\kappa[A],\in,A\rangle$} if the following statements hold for all $\gamma\in I$: 
 \begin{enumerate}[leftmargin=0.9cm]
  \item $\langle\LL_\gamma[A\cap\gamma],\in,A\cap\gamma\rangle$ is an elementary substructure of $\langle\LL_\kappa[A],\in,A\rangle$. 
  
  \item  $I\setminus\gamma$ is a set of indiscernibles for the structure $\langle\LL_\kappa[A],\in,A,\xi\rangle_{\xi<\gamma}$. 
 \end{enumerate}
Then {\cite[Lemma 1.2]{MR645907}} shows that a cardinal $\kappa$ is Ramsey if and only if for every $A\subseteq\kappa$, there is a good set of indiscernibles for $\langle\LL_\kappa[A],\in,A\rangle$. 
In the following, we will show that cardinals with $\mathbf{\Sigma}_1$-club property are Ramsey with respect to  subsets of $\kappa$ whose singletons are $\Sigma_1$-definable, in the sense that the club property implies the existence of good sets of indiscernibles for the corresponding structures $\langle\LL_\kappa[A],\in,A\rangle$.  
Moreover, we will show that the this restricted form of Ramseyness is downwards absolute to the Dodd--Jensen core model $\KK^{DJ}$ and, in this inner model, it is equivalent to the $\mathbf{\Sigma}_1$-club property. 
These arguments will also allow us to show that the existence of a cardinal with the $\mathbf{\Sigma}_1$-club property has much higher consistency strength than the existence of a cardinal with the $\mathbf{\Sigma}_1$-colouring property.

 \begin{proposition}\label{proposition:ClubPropertyGoodIndi}
  Let $\kappa$ be an uncountable regular cardinal and let $A$ be a subset of $\kappa$ such that the set $\{A\}$ is $\Sigma_n(\kappa,z)$-definable for some set $z$. If $\kappa$ has the $\Sigma_n(z)$-club property, then there is a club subset of $\kappa$ that is a good set of indiscernibles for $\langle\LL_\kappa[A],\in,A\rangle$. 
 \end{proposition}
 
 \begin{proof}
  Given an $\calL_\in$-formula $\varphi(v_0,\ldots,v_{k+m-1})$ and ordinals $\beta_0,\ldots,\beta_{m-1}<\kappa$, we define a function $\map{c_{\varphi,\beta_0,\ldots,\beta_{m-1}}}{[\kappa]^k}{2}$ by setting $$c_{\varphi,\beta_0,\ldots,\beta_{m-1}}(\{\alpha_0,\ldots,\alpha_k\})=0 ~ \Longleftrightarrow ~ \langle\LL_\kappa[A],\in,A\rangle\models\varphi(\alpha_0,\ldots,\alpha_k,\beta_0,\ldots,\beta_{m-1})$$ for all $\alpha_0<\ldots<\alpha_k<\kappa$. By our assumptions on $A$, we know that the set $\{\LL_\kappa[A]\}$ is $\Sigma_n(\kappa,z)$-definable and hence $c_{\varphi,\beta_0,\ldots,\beta_{m-1}}$ is a $\Sigma_n(z)$-partition. But this shows that there is $c_{\varphi,\beta_0,\ldots,\beta_{m-1}}$-homogeneous club in $\kappa$. Given $\beta<\kappa$, this implies that there is a club $C_\beta$ in $\kappa$ that is  $c_{\varphi,\beta_0,\ldots,\beta_{m-1}}$-homogeneous for every $\calL_\in$-formula $\varphi(v_0,\ldots,v_{k+m-1})$ and all $\beta_0,\ldots,\beta_{m-1}<\beta$. Let $C$ denote the intersection of $\bigtriangleup\Set{C_\beta}{\beta<\kappa}$  with the club of all $\gamma<\kappa$ such that $\langle\LL_\gamma[A\cap\gamma],\in,A\cap\gamma\rangle$ is an elementary substructure of $\langle\LL_\kappa[A],\in,A\rangle$.  Then it is easy to check that $\Lim(C)$ is a good set of indiscernibles for $\langle\LL_\kappa[A],\in,A\rangle$. 
 \end{proof}

The following corollary uses the above result to show that strong anti-large cardinal assumptions imply the existence of simply definable bistationary subsets of uncountable regular cardinals. In particular, it shows that the existence of a cardinal with the $\mathbf{\Sigma}_1$-club property implies the existence of $x^\#$ for every $x\in\RRR$.

 \begin{corollary}\label{corollary:ClubPropSharps}
  If $x$ is a real such that $x^\#$ does not exist and $\kappa$ is an uncountable regular cardinal, then there is a bistationary subset $A$ of $\kappa$ with the property that the set $\{A\}$ is $\Sigma_1(\kappa,\gamma_0,\ldots,\gamma_m,x)$-definable for some $\gamma_0,\ldots,\gamma_{m-1}<\kappa$. 
 \end{corollary}
 
\begin{proof}
 By our assumption, standard arguments (see, for example, {\cite[Theorem 9.14]{MR1994835}}) show that there is no uncountable good set of indiscernibles for $\langle\LL_\kappa[x],\in,x\rangle$ and hence Proposition \ref{proposition:ClubPropertyGoodIndi} shows that $\kappa$ does not have the $\Sigma_1(x)$-club property. 
Lemma \ref{lemma:CharacterizationsClubProperty} then yields the desired conclusion. 
\end{proof}


In \cite{MR2830415}, Gitman provided another useful characterization of Ramseyness by showing that a cardinal $\kappa$ is Ramsey if and only if for every $A\subseteq\kappa$, there is a weak $\kappa$-model $M$ with $A\in M$ and a weakly amenable countably complete $M$-ultrafilter on $\kappa$ (see {\cite[Proposition 2.8.(3)]{MR2830415}}). 
In combination with arguments from \cite{MR2830415}, the above results already show that the $\mathbf{\Sigma}_1$-club property implies the restriction of the above property to $\Sigma_1$-definable singletons. 
 We will later show that, in canonical inner models, this restricted property is actually equivalent to the $\mathbf{\Sigma}_1$-club property. This will allow us to show that the $\mathbf{\Sigma}_1$-club property is downwards absolute to the Dodd--Jensen core model $\KK^{DJ}$. 
The proof of the forward direction of {\cite[Proposition 2.8.(3)]{MR2830415}} in {\cite[Section 4]{MR2830415}} also provides a proof of the following statement.

\begin{lemma}\label{lemma:GitmanLemma}
   Let $\kappa$ be an uncountable regular cardinal and let $A$ be a subset of $\kappa$ with the property that $\kappa$ is an inaccessible cardinal in $\LL[A]$. If there is a good set of indiscernibles for $\langle\LL_\kappa[A],\in,A\rangle$, then there is a weak $\kappa$-model $M$ with $A\in M$ and a weakly amenable countably complete $M$-ultrafilter on $\kappa$. 
\end{lemma}

By combining Corollary \ref{corollary:InaccessibleInLA}, Proposition \ref{proposition:ClubPropertyGoodIndi}  and the above lemma, we directly obtain the following corollary.

\begin{corollary}
 Let $\kappa$ be an uncountable regular cardinal and let $A$ be a subset of $\kappa$ with the property that the set $\{A\}$ is $\Sigma_n(\kappa,z)$-definable for some set $z$. If $\kappa$ has the $\Sigma_n(z)$-club property, then there exist a weak $\kappa$-model $M$ with $A\in M$ and a weakly amenable countably complete $M$-ultrafilter on $\kappa$. \qed 
\end{corollary}

The next result will later allow us to show that, in the case $n=1$, the converse of the above implication also holds true in certain canonical inner models. The arguments used in its proof are taken from the proof of {\cite[Lemma 6.7]{MR3694344}}.

\begin{lemma}\label{lemma:IterableModelsClubProperty}
 Let $\kappa$ be an uncountable regular cardinal, let $z\in\HH{\kappa}$ and let $\varphi(v_0,\ldots,v_3)$ be a $\Sigma_0$-formula. Assume that there is a unique subset $A$ of $\kappa$ with the property that $\exists x ~ \varphi(A,x,\kappa,z)$ holds. If there exist a weak $\kappa$-model $M$ with the property that $A,\tc{\{z\}}\in M\models\exists x ~ \varphi(A,x,\kappa,z)$ and a weakly amenable $M$-ultrafilter $U$ on $\kappa$ such that $\langle M,\in,U\rangle$ is $\omega_1$-iterable, then $A$ either contains a club subset of $\kappa$ or is disjoint from such a set. 
\end{lemma}

\begin{proof}
 Fix $a\in M$ such that $\varphi(A,a,\kappa,z)$ holds and pick an elementary submodel $\langle N,\in,F\rangle$ of $\langle M,\in,U\rangle$ of cardinality less than $\kappa$ with $\tc{\{z\}}\cup\{A,a\}\subseteq N$. Let $\map{\pi}{N}{M_0}$ denote the corresponding transitive collapse and set $U_0=\pi[F]$. 
Then $U_0$ is a weakly amenable $M_0$-ultrafilter on $\pi(\kappa)$ and {\cite[Theorem 19.15]{MR1994835}} implies that $\langle M_0,\in,U_0\rangle$ is $\omega_1$-iterable. Let $$\langle\seq{M_\alpha}{\alpha\leq\kappa},\seq{\map{j_{\alpha,\beta}}{M_\alpha}{M_\beta}}{\alpha\leq\beta\leq\kappa}\rangle$$ denote the corresponding iteration of length $\kappa+1$. Then we have $(j_{0,\kappa}\circ\pi)(\kappa)=\kappa$, $(j_{0,\kappa}\circ\pi)(z)=z$ and $\Sigma_1$-upwards absoluteness implies that $\exists x ~ \varphi((j_{0,\kappa}\circ\pi)(A),x,\kappa,z)$ holds. This allows us to conclude that $A=(j_{0,\kappa}\circ\pi)(A)$ and $(j_{0,\alpha}\circ\pi)(A)=A\cap(j_{0,\alpha}\circ\pi)(\kappa)$ for all $\alpha<\kappa$.  Define $C$ to be the club $\Set{(j_{0,\alpha}\circ\pi)(\kappa)}{\alpha<\kappa}$ in $\kappa$. First, assume that $A\in U$. Then $(j_{0,\alpha}\circ\pi)(A)\in U_\alpha$ and hence $(j_{0,\alpha}\circ\pi)(\kappa)\in (j_{0,\alpha+1}\circ\pi)(A)\subseteq A$ for all $\alpha<\kappa$. This shows that $C$ is a subset of $A$ in this case. In the other case, if $A\notin U$, then the same argument shows that $A\cap C=\emptyset$.   
\end{proof}

\begin{lemma}\label{lemma:GoodWOIterableModelsClubProp}
 If $\kappa$ is an uncountable regular cardinal such that there exists a good $\Sigma_1(\kappa,y)$-well-ordering of $\POT{\kappa}$ for some $y\in\HH{\kappa}$, then the following statements are equivalent: 
 \begin{enumerate}[leftmargin=0.9cm]
  \item $\kappa$ has the $\mathbf{\Sigma}_1$-club property. 

  \item For all $A\subseteq\kappa$ with the property that the set $\{A\}$ is $\Sigma_1(\kappa,z)$-definable for some $z\in\HH{\kappa}$, there is a weak $\kappa$-model $M$ with $A\in M$ and a weakly amenable countably complete $M$-ultrafilter on $\kappa$. 
 \end{enumerate}
\end{lemma}

\begin{proof}
  Let $\lhd$ be a well-ordering of some class containing $\POT{\kappa}$ such that the class $I(\lhd)$  is $\Sigma_1(\kappa,z)$-definable. 
 Fix a $\Sigma_0$-formula $\varphi(v_0,\ldots,v_3)$ and $z\in\HH{\kappa}$ such that there is a unique subset $A$ of $\kappa$ with the property that $\exists x ~ \varphi(A,x,\kappa,z)$ holds. Then there is an $x\in\HH{\kappa^+}$ such that $\varphi(A,x,\kappa,z)$ holds. 
 Let $B$ denote the $\lhd$-least element of $\POT{\kappa}$ with the property that, if $\alpha>\kappa$ is minimal with $\LL[B]\models\ZFC^{-}$, then we have $A,\tc{\{z\}}\in\LL_\alpha[B]\models\exists x ~ \varphi(A,x,\kappa,z)$. 
Then the set $\{B\}$ is $\Sigma_1(\kappa,y,z)$-definable and our assumptions yield a weak $\kappa$-model $M$ with $B\in M$ and a weakly amenable countably complete $M$-ultrafilter on $\kappa$. 
Then $A,z\in M\models\exists x ~ \varphi(A,x,\kappa,z)$ and, since countable completeness implies $\omega_1$-iterability (see {\cite[Lemma 19.11 and 19.12]{MR1994835}}), we can apply Lemma \ref{lemma:IterableModelsClubProperty} to conclude that $A$ either contains a club in $\kappa$ or is disjoint from such a set. 
\end{proof}

\begin{lemma}
 If $\kappa$ is an uncountable regular cardinal with the $\mathbf{\Sigma}_1$-club property, then $\kappa$ is an inaccessible cardinal with the $\mathbf{\Sigma}_1$-club property in $\KK^{DJ}$.
\end{lemma}

\begin{proof}
 Fix $z\in\HH{\kappa}^{\KK^{DJ}}$ and $A\in\POT{\kappa}^{\KK^{DJ}}$ such that the set $\{A\}$ is $\Sigma_1(\kappa,z)$-definable in $\KK^{DJ}$. 
  By Corollary \ref{corollary:ClubPropSharps}, our assumption implies the existence of $0^\#$ and hence results of Dodd and Jensen (see {\cite[p. 238]{MR730856}}) show that $\KK^{DJ}$ is equal to the union of all \emph{lower parts} of \emph{iterable premice} in this situation. 
Since the class of all iterable premice is $\Sigma_1(\kappa)$-definable (see, for example, the proof of {\cite[Lemma 2.3]{GoodWO}}), this shows that the class $\KK^{DJ}$ is also $\Sigma_1(\kappa)$-definable in this case and we can conclude that the set $\{A\}$ is $\Sigma_1(\kappa,z)$-definable in $\VV$. 
Therefore, we can apply Proposition \ref{proposition:ClubPropertyGoodIndi} to find a club subset of $\kappa$ that is a good set of indiscernibles for $\langle\LL_\kappa[B],\in,B\rangle$. 
In this situation, we can apply the \emph{Jensen Indiscernibles Lemma} (see {\cite[Lemma 1.3]{MR645907}}) to find a good set of indiscernibles for $\langle\LL_\kappa[B],\in,B\rangle$ of cardinality $\kappa$ that is an element of $\KK^{DJ}$. 
Since Corollary \ref{corollary:InaccessibleInLA} shows that $\kappa$ is inaccessible in $\LL[B]$, we can now apply Lemma \ref{lemma:GitmanLemma} to show that in $\KK^{DJ}$, there is a weak $\kappa$-model $M$ with $A\in M$ and a weakly amenable countably complete $M$-ultrafilter on $\kappa$. 
But now, the results of {\cite[Section 2]{GoodWO}} show that the restriction of the canonical well-ordering of $\KK^{DJ}$ to $\POT{\kappa}^{\KK^{DJ}}$ is a good $\Sigma_1(\kappa)$ in $\KK^{DJ}$. 
Therefore, the above computations allow us to use Lemma \ref{lemma:GoodWOIterableModelsClubProp} in $\KK^{DJ}$ to conclude that $\kappa$ has the $\mathbf{\Sigma}_1$-club property in this model. 
Finally, we can apply Proposition  \ref{proposition:FailureSigma1PPgoodWO} to show that $\kappa$ is inaccessible in $\KK^{DJ}$.  
\end{proof}

Next, we provide an upper bounds for the consistency strength of the existence of an inaccessible cardinal with the $\mathbf{\Sigma}_1$-club property  with the help of the following large cardinal property strengthening weak compactness that was introduced by Sharpe and Welch in \cite{MR2817562} and extensively studied  in \cite{MR2830435}. 
 %
 %
 %

\begin{definition}\label{definition:IterableCardinal}
 An uncountable cardinal $\kappa$ is \emph{$\omega_1$-iterable} if for every subset $A$ of $\kappa$, there is a weak $\kappa$-model $M$ and a weakly amenable $M$-ultrafilter $U$ on $\kappa$ such that $A\in M$ and $\langle M,\in,U\rangle$ is $\omega_1$-iterable. 
\end{definition}

The following corollary is a direct consequence of Lemma \ref{lemma:IterableModelsClubProperty}. 
Note that this result is basically already proven in {\cite[Section 6]{MR3694344}}, but only for $\Sigma_1(\kappa)$-definitions.

\begin{corollary}\label{proposition:IterableCardinalsClubProperty}
  All $\omega_1$-iterable cardinals have the $\mathbf{\Sigma}_1$-club property. \qed 
\end{corollary}

In the following, we show that the $\mathbf{\Sigma}_1$-club property at $\omega_1$ can be established by collapsing an inaccessible cardinal with this property.

\begin{lemma}\label{lemma:SmallForcingClubProperty}
 Let $\kappa$ be an uncountable regular cardinal with the $\mathbf{\Sigma}_n$-club property and let $\PPP\in\HH{\kappa}$ be a partial order. If $G$ is $\PPP$-generic over $\VV$, then $\kappa$ has the $\mathbf{\Sigma}_n$-club property in $\VV[G]$. 
\end{lemma}

\begin{proof}
 Fix $z\in\HH{\kappa}^{\VV[G]}$, a $\Sigma_n(z)$-partition $\map{c}{[\kappa]^k}{\gamma}$ in $\VV[G]$ with $\gamma<\kappa$ and a $\Sigma_n$-formula $\varphi(v_0,\ldots,v_{k+2})$ defining $c$ in $\VV[G]$ as in Definition \ref{definition:SigmanPartition}.  
 Work in $\VV$ and fix a condition $p$ in $\PPP$, a $\PPP$-name $\dot{z}\in\HH{\kappa}^\VV$ with $z=\dot{z}^G$ and a bijection $b$ between a cardinal $\nu$ and the set of all conditions in $\PPP$ below $p$. 
 Given $\alpha_0<\ldots<\alpha_{k-1}<\kappa$, define $c_0(\{\alpha_0,\ldots,\alpha_{k-1}\})$ to be the least ordinal of the form $\goedel{\beta}{\delta}$, where $\beta<\nu$, $\delta<\gamma$ and $b(\beta)\Vdash_\PPP\varphi(\check{\alpha}_0,\ldots,\check{\alpha}_{k-1},\check{\delta},\check{\kappa},\dot{z})$. Then the arguments used in the proof of Lemma \ref{lemma:LevySolovayPartitionProperty} show that $c_0$ is a $\Sigma_n(\dot{z},\PPP)$-partition. 
 By our assumptions, genericity now yields $q\in G$, $\delta<\gamma$ and a club subset $C$ of $\kappa$ in $\VV$ with $q\Vdash^\VV_\PPP\varphi(\check{\alpha}_0,\ldots,\check{\alpha}_{k-1},\check{\delta},\check{\kappa},\dot{z})$ for all $\alpha_0,\ldots,\alpha_{n-1}\in C$ with $\alpha_0<\ldots<\alpha_{n-1}$. 
 In particular, there is a $c$-homogeneous subset of $\kappa$ in $\VV[G]$ that is closed and unbounded in $\kappa$.  
\end{proof}

\begin{lemma}\label{lemma:ForcingClubPropertyWeaklyInaccessibleSuccessor}
 Let $\kappa$ be an inaccessible cardinal, let $\PPP\in\{\Add{\omega}{\kappa},\Col{\omega}{{<}\kappa}\}$ and let $G$ be $\PPP$-generic over $\VV$. If $\kappa$ has the $\mathbf{\Sigma}_1$-club property in $\VV$, then $\kappa$ has the $\mathbf{\Sigma}_1$-club property in $\VV[G]$. 
\end{lemma}

\begin{proof}
 Pick a subset $A$ of $\kappa$ in $\VV[G]$ that is $\Sigma_n(\kappa,z)$-definable in $\VV[G]$ for some $z\in\HH{\kappa}^{\VV[G]}$. 
As in the proof of Lemma \ref{lemma:Sigma1PartitionPropertyCollapses}, we can use Proposition \ref{proposition:SmallForcingDef} to find a regular cardinal $\nu<\kappa$ in $\VV$ and $H\in\VV[G]$ such that $H$ is either $\Add{\omega}{\nu}$- or $\Col{\omega}{\nu}$-generic over $\VV$, $A,z\in\VV[H]$, the set $\{A\}$ is $\Sigma_n(\kappa,\mu,z)$-definable in $\VV[H]$ and $\VV[G]$ is a $\PPP$-generic extension of $\VV[H]$. 
 In this situation, Lemma \ref{lemma:SmallForcingClubProperty} shows that $\kappa$ has the $\mathbf{\Sigma}_1$-club property in $\VV[H]$ and hence there is a club subset $C$ of $\kappa$ in $\VV[H]$ that is either contained in $A$ or disjoint from $A$.  
\end{proof}

In the remainder of this section, we show that the validity of the $\mathbf{\Sigma}_2$-club property at $\omega_1$ is equiconsistent with the existence of a measurable cardinal.

\begin{proposition}
 If $\omega_1$ has the $\Sigma_2$-club property, then $\omega_1$ is a measurable cardinal in $\HOD$. 
\end{proposition}

\begin{proof}
 Let $F$ denote the intersection of the club filter on $\omega_1$ with $\HOD$. Then $F$ is an element of $\HOD$. 
 Assume that $F$ does not witness the measurability of $\omega_1$ in $\HOD$. By the closure properties of the club filter, this implies that $\HOD$ contains a bistationary subset of $\omega_1$. If $A$ denotes the least such subset in the canonical well-ordering of $\HOD$, then the fact that this ordering is a good $\Sigma_2$-well-ordering implies that the set $\{A\}$ is $\Sigma_2$-definable, contradicting our assumption. 
\end{proof}

The next lemma shows that a measurable cardinal is also an upper bound for the consistency of the validity of the $\mathbf{\Sigma}_2$-club property at $\omega_1$. Its proof is small variation of a classical result of Jech, Magidor, Mitchell and Prikry from \cite{MR560220}.

\begin{lemma}\label{lemma:ConsClub2Omega1}
 If $\kappa$ is a measurable cardinal, then there is a generic extension $\VV[G]$ of $\VV$ with the property that $\kappa=\omega_1^{\VV[G]}$ and no bistationary subset of $\omega_1$ in $\VV[G]$ is contained in $\HOD(\RRR)^{\VV[G]}$. In particular, in $\VV[G]$, the cardinal $\omega_1$ has the $\mathbf{\Sigma}_n$-club property for all $n<\omega$. 
\end{lemma}

\begin{proof}
 Let $U$ be a normal ultrafilter  on $\kappa$, let $\map{j}{\VV}{M}$ denote the canonical ultrapower embedding induced by $U$, let $G$ be $\Col{\omega}{{<}j(\kappa)}$-generic over $\VV$, let $G_0$ denote the filter on $\Col{\omega}{{<}\kappa}$ induced by $G$ and let $\map{j_G}{\VV[G_0]}{M[G]}$ denote the canonical lifting of $j$ to $\VV[G]$ (see {\cite[Proposition 9.1]{MR2768691}}). Since $\Col{\omega}{{<}\kappa}$ satisfies the $\kappa$-chain condition in $\VV$, we know that every element of $U$ is a stationary subset of $\kappa$ in $\VV[G_0]$. 
  
 Work in $\VV[G_0]$. By the above remark, if $A$ is an element of $U$, then the partial order $\CCC(A)$ consisting of all bounded closed subsets of $A$ ordered by end-extension is $\sigma$-distributive, weakly homogeneous and forces $A$ to contain a club subset of $\kappa=\omega_1^{\VV[G_0]}$ (see {\cite[Section 3.5, Theorem 1]{MR0373889}} and {\cite[Theorem 1]{MR716625}}). 
Let $\vec{\CCC}$ denote the countable support product of forcings of the form $\CCC(A)$ with $A\in U$. Then $\vec{\CCC}$ is weakly homogeneous and the fact that $\CH$ holds allows us to use a $\Delta$-system argument to show that $\vec{\CCC}$ satisfies the $\aleph_2$-chain condition.

\begin{claim*}
 $\vec{\CCC}$ is $\sigma$-distributive in $\VV[G_0]$. 
\end{claim*}

\begin{proof}[Proof of the Claim]
 Work in $\VV[G_0]$. Fix a condition $\vec{p}$ in $\vec{\CCC}$ and a $\vec{\CCC}$-nice name $\tau$ for a countable set of ordinals. 
Since $\vec{\CCC}$ satisfies the $\aleph_2$-chain condition, there is a subset $U_1$ of $U$ of cardinality $\kappa$ such that the support of $\vec{p}$ and the supports of all conditions appearing in $\tau$ are subsets of $U_1$. 
Using the fact that $\Col{\omega}{{<}\kappa}$ satisfies the $ \kappa$-chain condition in $\VV$, we find a subset $U_0$ of $U$ in $\VV$ that contains $U_1$ and has cardinality $\kappa$ in $\VV$. 
In this situation, the closure properties of $M$ imply that the sets  $U_0$, $j[U_0]$ and $j\restriction U_0$ are all contained in $M$ and all three sets are countable in $M[G]$. 
Let $\vec{\CCC}_0$ denote the countable support product of all partial orders of the form $\CCC(A)$ with $A\in U_0$, let $\vec{p}_0$ denote the condition in $\vec{\CCC}_0$ corresponding to $\vec{p}$ and let $\tau_0$ denote the canonical $\vec{\CCC}_0$-name induced by $\tau$.   

Since $\vec{\CCC}_0\in\HH{\kappa^+}^{\VV[G_0]}$ and $\Col{\omega}{{<}\kappa}$ satisfies the $\kappa$-chain condition in $\VV$, there is a $\Col{\omega}{{<}\kappa}$-name $\dot{\CCC}\in\HH{\kappa^+}^\VV$ for a partial order with the property that $\dot{\CCC}^{G_0}$ is the suborder of $\vec{\CCC}_0$ consisting of all conditions below $\vec{p}_0$. But then  $\dot{\CCC}$ is an element of $M$ and, by {\cite[Theorem 14.2]{MR2768691}}, there is a complete embedding $\map{\iota}{\Col{\omega}{{<}\kappa}*\dot{\CCC}}{\Col{\omega}{{<}j(\kappa)}}$ in $M$ that extends the identity on $\Col{\omega}{{<}\kappa}$. 
Let $\vec{H}$ denote the filter on $\vec{\CCC}_0$ induced by $\iota$ and $G$. Moreover, given $A\in U_0$, let $H_A$ denote the filter on $\CCC(A)$ induced by $\vec{H}$. 
For each $A\in U_0$, we then have $\bigcup H_A\in M[G]$, $\kappa\in j(A)$, $\bigcup H_A\subseteq A\subseteq j(A)$ and hence $\{\kappa\}\cup\bigcup H_A$ is a bounded closed subset of $j(A)$ in $M[G]$. 
By the above computations, there is a condition $\vec{q}$ in $j_G(\vec{\CCC}_0)$ with support $j[U_0]$ and the property that $\vec{q}\vspace{0.7pt}(j(A))=\{\kappa\}\cup\bigcup H_A$ for all $A\in U_0$.
 But then we have $\vec{q}\leq_{j_G(\vec{\CCC}_0)}j_G(\vec{r})$ for all $\vec{r}\in\vec{H}$. In particular, if $n<\omega$, then there is a condition $\vec{r}_n$ in $\vec{H}$ that decides the $n$-the element of $\tau_0$ in $\VV[G]$ and satisfies $\vec{q}\leq_{j_G(\vec{\CCC}_0)}j_G(\vec{r}_n)$.  
By elementarity, this yields a condition $\vec{r}$ below $\vec{p}$ in $\vec{\CCC}_0$ with $$\vec{r}\Vdash^{\VV[G_0]}_{\vec{\CCC}_0}\anf{\dot{\tau}_0=\check{c}}$$ for some countable set  of ordinals $c$. Since $\vec{\CCC}_0$ is a complete suborder of $\vec{\CCC}$, these computations yield the statement of the claim.  
\end{proof}

Now, let $\vec{H}$ be $\vec{\CCC}$-generic over $\VV[G_0]$, fix a subset $B$ of $\kappa$ in $\HOD(\RRR)^{\VV[G_0,\vec{H}]}$ and pick $x\in\RRR^{\VV[G_0,\vec{H}]}$ with $B\in\HOD_x^{\VV[G_0,\vec{H}]}$. By the above claim, we have $\kappa=\omega_1^{\VV[G_0,\vec{H}]}$, $x\in\VV[G_0]$ and the homogeneity of $\vec{\CCC}$ in $\VV[G]$ implies that $B\in\VV[G_0]$. Moreover, since $\vec{\CCC}$ is definable in $\VV[G_0]$ by a formula that only uses $U$ as a parameter, we know that $B\in\HOD_{U,x}^{\VV[G_0]}$. Since $\VV[G_0]$ is a $\Col{\omega}{{<}\kappa}$-generic extension of $\VV[x]$, the homogeneity of $\Col{\omega}{{<}\kappa}$ of $\Col{\omega}{{<}\kappa}$ implies that $B\in\VV[x]$. In this situation, standard arguments show that $\VV[x]$ is a generic extension of $\VV$ using a partial order of size less than $\kappa$ in $\VV$ and therefore the proof of the Levy--Solovay--Theorem (see, for example, {\cite[Proposition 10.15]{MR1994835}}) shows that the set $\Set{E\in\POT{\kappa}^{\VV[x]}}{\exists D\in U ~ D\subseteq E}$ witnesses the measurability of $\kappa$ in $\VV[x]$. By the above computations, this yields an $A\in U$ such that either $A\subseteq B$ or $A\cap B=\emptyset$ holds. But now, our constructions ensure that there is a club subset $C$ of $\kappa$ in $\VV[G_0,\vec{H}]$ with $C\subseteq A$ and therefore $A$ is not a bistationary subset of $\omega_1$ in $\VV[G_0,\vec{H}]$. 
\end{proof}


\section{Definable partitions of countable ordinals}\label{section:DefPartCountableOrdinals}

In this short section, we show that many natural extensions of the axioms of $\ZFC$ cause $\omega_1$ to have strong partition properties for simply definable colourings. These results are summarized in the following theorem.

\begin{theorem}\label{theorem:Omega1Sigma1wc}
 Assume that one of the following assumptions holds: 
 \begin{enumerate}[leftmargin=0.9cm]
  \item There is a measurable cardinal above a Woodin cardinal.
  
  \item There is a measurable cardinal and a precipitous ideal on $\omega_1$. 
  
  \item \emph{Bounded Martin's Maximum} $\BMM$ holds and the nonstationary ideal on $\omega_1$ is precipitous.  
  
  \item Woodin's Axiom $(*)$ holds. 
 \end{enumerate}
 Then $\omega_1$ has the $\mathbf{\Sigma}_1$-club property. 
\end{theorem}

We prove the last implication stated above by providing an alternative way to establish the conclusion of Lemma \ref{lemma:ConsClub2Omega1}.

\begin{proposition}\label{proposition:PmaxClubn}
 Assume that $\AD$ holds in $\LL(\RRR)$ and $G$ is $\PPP_{max}$-generic over $\LL(\RRR)$. Then $\omega_1$ has the $\mathbf{\Sigma}_2$-club property in $\LL(\RRR)[G]$. 
\end{proposition}

\begin{proof}
 Pick a subset $A$ of $\omega_1$ in $\HOD(\RRR)^{\LL(\RRR)[G]}$. Since $\PPP_{max}$ is $\sigma$-closed and weakly homogeneous in $\LL(\RRR)$ (see {\cite[Lemma 2.10 \& 3.4]{MR2768703}}), we know that $A$ is an element of $\LL(\RRR)$. By a classical result of Solovay, this shows that $A$ either contains a club subset of $\omega_1$ in $\LL(\RRR)$ or is disjoint from such a set. 
\end{proof}

The above statement directly yields the fourth implication of Theorem \ref{theorem:Omega1Sigma1wc}.

\begin{proof}[Proof of implication (iv) of Theorem \ref{theorem:Omega1Sigma1wc}]
 Assume that Woodin's Axiom $(*)$ holds, i.e. $\AD$ holds in $\LL(\RRR)$ and there is some $G$ that is $\PPP_{max}$-generic over $\LL(\RRR)$ and satisfies $\POT{\omega_1}\subseteq\LL(\RRR)[G]$. Fix $z\in\HH{\omega_1}$ and $A\in\POT{\omega_1}$ such that the set $\{A\}$ is $\Sigma_1(\omega_1,z)$-definable. Then the same formula defines $\{A\}$ in $\LL(\RRR)[G]$ and Proposition \ref{proposition:PmaxClubn} implies that $A$ either contains a club or is disjoint from such a subset. 
\end{proof}

We now derive the other implications of Theorem \ref{theorem:Omega1Sigma1wc} from the results of \cite{MR3694344}.

\begin{proof}[Proof of the implications (i)--(iii) of Theorem \ref{theorem:Omega1Sigma1wc}]
 Note that the results of \cite{MR1300637} and \cite{MR1257469} imply that (i) implies that $M_1^\#(A)$ exists for every subset $A$ of $\omega_1$ (see {\cite[p. 1738]{MR2768699}} and {\cite[p. 1660]{MR2768698}}). Moreover, {\cite[Theorem 2.1]{MR3694344}} shows that (ii) implies the same conclusion. 
 Now, assume, towards a contradiction, that one of the first three assumptions listed in Theorem \ref{theorem:Omega1Sigma1wc} holds and $\omega_1$ does not have the $\mathbf{\Sigma}_1$-club property. Then Lemma  \ref{lemma:CharacterizationsClubProperty} yields a bistationary subset $A$ of $\omega_1$ with the property that the set $\{A\}$ is $\Sigma_1(\omega_1,z)$-definable for some $z\in\HH{\omega_1}$.  
 By the above remarks, we can now apply {\cite[Lemma 4.11]{MR3694344}} to conclude that the set $\{A\}$ contains both an element of the club filter and the non-stationary ideal, a contradiction. 
\end{proof}

We end this section by showing that the existence of a Woodin cardinal alone does not cause $\omega_1$ to have the $\Sigma_1$-colouring property.

 \begin{corollary}\label{corollary:FailureSigma1wcInM1}
 If $M_1$ exists, then, in $M_1$, the cardinal $\omega_1$ does not have the $\Sigma_1$-colouring property.  
\end{corollary}

\begin{proof}
 By {\cite[Theorem 5.2]{MR3694344}}, there is a good $\Sigma_1(\omega_1)$-well-ordering of $\POT{\omega_1}$ in $M_1$. Therefore, Proposition \ref{proposition:FailureSigma1PPgoodWO} yields the conclusion of the corollary. 
\end{proof}


\section{Definable partitions of $[\omega_2]^2$}\label{section:Omega2}

In this section, we study simply definable colourings of finite subsets of the second uncountable cardinal and the influence of canonical extensions of $\ZFC$ on these partitions. 
A combination of Corollary \ref{corollary:GenericTreeCodingCounterexamples} with Lemma \ref{lemma:Sigma1PartitionPropertyCollapses} already shows that  the statement that $\omega_2$ has the $\mathbf{\Sigma}_n$-colouring property is independent from the axiom of $\ZFC$ together with large cardinal assumptions for all $0<n<\omega$.

The following proposition shows that strong forcing axioms outright imply a failure of the $\mathbf{\Sigma}_1$-colouring  property at $\omega_2$.

\begin{proposition}\label{proposition:BPFAfailureBoldface}
 Assume that the \emph{Bounded Proper Forcing Axiom $\BPFA$} holds. If $z\subseteq\omega_1$ with $\omega_1=\omega_1^{\LL[z]}$, then $\omega_2$ does not have the $\Sigma_1(z)$-colouring property.  
\end{proposition}

\begin{proof}
 By our assumption, there is a \emph{ladder system} $\vec{C}$ (i.e. a sequence $\seq{C_\alpha}{\alpha\in\Lim\cap\omega_1}$ with the property that $C_\alpha$ is a cofinal subset of $\alpha$ of order-type $\omega$ for every countable limit ordinal $\alpha$) with the property that the set $\{\vec{C}\}$ is $\Sigma_1(\omega_1,z)$-definable. 
 By {\cite[Theorem 2]{MR2231126}}, $\BPFA$ implies that $\HH{\omega_2}$ has a good $\Sigma_1(\vec{C})$-well-ordering. 
 Since this implies that $\HH{\omega_2}$ has a good $\Sigma_1(\omega_2,z)$-well-ordering, we can apply Proposition \ref{proposition:FailureSigma1PPgoodWO} to conclude that $\omega_2$ does not have the $\Sigma_1(z)$-colouring property. 
\end{proof}

 The failures of the $\mathbf{\Sigma}_1$-colouring property provided by Corollary \ref{corollary:GenericTreeCodingCounterexamples} and Proposition \ref{proposition:BPFAfailureBoldface} both make use of subsets of $\omega_1$ that encode a great amount of information. 
Therefore, it is natural to consider even simpler partitions and ask if large cardinal assumptions or strong forcing axioms cause $\omega_2$ to possess the $\Sigma_1$-colouring property. 
This question is answered negatively by the following result. It also answers one of the main questions left open by the results of \cite{MR3694344}.

\begin{theorem}\label{theorem:MainOmega2Negative}
 Assume that $\BPFA$ holds. If there is a $\lhd$ well-ordering of the reals that is definable over the structure $\langle\HH{\omega_2},\in\rangle$ by a formula with parameter $z\in\HH{\omega_2}$, then the following statements hold:
 \begin{enumerate}[leftmargin=0.9cm]
  \item The well-ordering $\lhd$ is $\Sigma_1(\omega_2,z)$-definable. 

  \item The cardinal $\omega_2$ does not have the $\Sigma_1(z)$-colouring property.  
 \end{enumerate}
\end{theorem}

 Results of Asper\'o and Larson cited below show that  the statement that the assumptions of Theorem \ref{theorem:MainOmega2Negative} are satisfied for the empty set as a parameter are compatible with both large cardinal assumptions and strong forcing axioms. In particular, the existence of a $\Sigma_1(\omega_2)$-definable well-ordering of the reals is compatible with these assumptions. This answers {\cite[Question 7.5]{MR3694344}}.

\begin{remark}
 \begin{enumerate}[leftmargin=0.9cm]
  \item If $\kappa$ is supercompact, then {\cite[Theorem 5.2]{MR2320944}} shows that there is a semi-proper partial order $\PPP\subseteq\HH{\kappa}$ with the property that whenever $G$ is $\PPP$-generic over $\VV$, then $\PFA^{++}$ (see {\cite[Definition 5.1]{MR2320944}}) holds in $\VV[G]$ and there is a well-ordering of $\HH{\omega_2}^{\VV[G]}$ that is   definable over $\langle\HH{\omega_2},\in\rangle$ by a formula without parameters. 
  
  \item If $\kappa$ is a supercompact limit of supercompact cardinals, then {\cite[Theorem 7.1]{MR2474445}} yields a semi-proper partial order $\PPP$ with the property that whenever $G$ is $\PPP$-generic over $\VV$, then $\MM^{{+}\omega}$ (see {\cite[Section 1]{MR2474445}}) holds in $\VV[G]$ and there is a well-ordering of $\HH{\omega_2}^{\VV[G]}$ that is definable over $\langle\HH{\omega_2},\in\rangle$ by a formula without parameters. 
 \end{enumerate}
\end{remark}

Note that it is not know whether the stronger forcing axiom $\MM^{++}$ (see {\cite[Section 1]{MR2474445}}) implies that no well-ordering of the reals is definable over $\langle\HH{\omega_2},\in\rangle$ by a formula without parameters. This question is motivated by the open question whether $\MM^{++}$ implies Woodin's axiom $(*)$ (see \cite{MR2474445} for a discussion).

Theorem \ref{theorem:MainOmega2Negative} is a direct consequence of the following lemma. The lemma itself follows directly from arguments due to Caicedo and  Veli{\v{c}}kovi{\'c} that are used to prove {\cite[Theorem 1]{MR2231126}} stating that, if $\BPFA$ holds and $M$ is an inner model of $\ZFC+\BPFA$ with $\omega_2=\omega_2^M$, then $\POT{\omega_1}\subseteq M$.

\begin{lemma}\label{lemma:BPFALocallyDefinable}
 If $\BPFA$ holds, then the set $\{\HH{\omega_2}\}$ is $\Sigma_1(\omega_2)$-definable. 
\end{lemma}

\begin{proof}
 The proof of {\cite[Theorem 1]{MR2231126}} shows that there is a finite fragment $\mathsf{F}$ of the theory $\ZFC+\BPFA$ with the property that $\ZFC+\BPFA$ proves that every transitive model $M$ of $\mathsf{F}+\anf{\textit{$\omega_2$ exists}}$ with $\omega_2=\omega_2^M$ contains all subsets of $\omega_1$. This shows that $\BPFA$ implies that $\HH{\omega_2}$ is the unique set $B$ with the property that there is a transitive model $M$ of $\mathsf{F}+\anf{\textit{$\omega_2$ exists}}$ with $\omega_2=\omega_2^M$ and $B=\HH{\omega_2}^M$. In particular, $\BPFA$ implies that the set $\{\HH{\omega_2}\}$ is $\Sigma_1(\omega_2)$-definable. 
\end{proof}

Note that both large cardinal assumptions and strong forcing axioms imply that the set $\HH{\omega_1}$ is not $\Sigma_1(\omega_1)$-definable. This follows directly from the fact that there are non-projective sets of reals that can be defined over the structure $\langle\HH{\omega_1},\in\rangle$ and {\cite[Lemma 3.3]{MR3694344}} showing that these extensions of $\ZFC$ imply that every $\Sigma_1(\omega_1)$-definable set of reals is $\Sigma^1_3$-definable.

\begin{proof}[Proof of Theorem \ref{theorem:MainOmega2Negative}]
  Assume that $\BPFA$ holds and that there is a well-ordering $\lhd$ of ${}^\omega 2$ that is definable over the structure $\langle\HH{\omega_2},\in\rangle$ by a formula with parameter $z\in\HH{\omega_2}$. Then Lemma \ref{lemma:BPFALocallyDefinable} directly implies that the set $\lhd$ is $\Sigma_1(\omega_2,z)$-definable. Since $\BPFA$ implies that $\CH$ fails, we know that $\langle{}^\omega 2,\lhd\rangle$ has order-type at least $\omega_2$. Let $\map{\iota}{\omega_2}{{}^\omega 2}$ denote the canonical enumeration of the first $\omega_2$-many elements of $\langle{}^\omega 2,\lhd\rangle$. Then $\iota$ is also definable over the structure $\langle\HH{\omega_2},\rangle$ by a formula with parameter $z$. In this situation, Lemma  \ref{lemma:BPFALocallyDefinable} shows that $\iota$ is $\Sigma_1(\omega_2,z)$-definable and we can apply Proposition \ref{Proposition:NoDefinableInjection} to conclude that $\omega_2$ does not have the $\Sigma_1(z)$-colouring  property.  
\end{proof}


\section{Limit cardinals}\label{section:Limits}

We now consider the question, which of the above partition relations for definable colourings can hold at regular limit cardinals. 
In these considerations, we focus on inaccessible cardinals that are not weakly compact. 
 We start by showing that there are many such cardinals with the $\mathbf{\Sigma}_1$-colouring property below weakly compact cardinals and there are many such cardinals with the $\mathbf{\Sigma}_1$-club property below weakly compact cardinals with this property.

\begin{theorem}\label{theorem:Sigma1WCbelowWC}
 Let $\kappa$ be a weakly compact cardinal, let $A$ be a subset of $\kappa$ and let $\Psi(v)$ be a $\Pi^1_1$-formula with $\VV_\kappa\models\Psi(A)$. 
 \begin{enumerate}[leftmargin=0.9cm]
  \item The statement $\Psi(A)$ reflects to an inaccessible cardinal less than $\kappa$ with the $\mathbf{\Sigma}_1$-colouring property. 
  
  \item If the cardinal $\kappa$ has the $\mathbf{\Sigma}_1$-club property, then the statement $\Psi(A)$ reflects to an inaccessible cardinal less than $\kappa$ with the $\mathbf{\Sigma}_1$-club property. 
 \end{enumerate}
\end{theorem}

\begin{proof}
 Pick an elementary submodel $M$ of $\HH{\kappa^+}$ of cardinality $\kappa$ with $\kappa,A\in M$ and ${}^{{<}\kappa}M\subseteq M$. The results of {\cite{MR1133077}} now yield a transitive set $N$ and an elementary  embedding $\map{j}{M}{N}$ with critical point $\kappa$ such that both $M$ is an element of $N$.  Then $\kappa$ is inaccessible in $N$, $\HH{\kappa}\subseteq M\subseteq N$, $A=j(A)\cap\kappa\in N$ and $\Pi^1_1$-downwards absoluteness implies that $(\VV_\kappa\models\Psi(A))^N$.

\begin{claim*}
 $\Sigma_1$-formulas with parameters in $M$ are absolute between $M$ and $N$. 
\end{claim*}

\begin{proof}[Proof of the Claim]
 Since $M\subseteq N$, it suffices to show that $\Sigma_1$-formulas with parameters from $M$ are downwards-absolute from $N$ to $M$. Fix $z_0,\ldots,z_{n-1}\in M$ and a $\Sigma_1$-formula $\varphi(v_0,\ldots,v_{n-1})$ such that $\varphi(z_0,\ldots,z_{n-1})$ holds in $N$. Then $\Sigma_1$-upwards absoluteness implies that $\varphi(z_0,\ldots,z_{n-1})$ holds in $\VV$ and the $\Sigma_1$-Reflection Principle implies that this statement holds in $\HH{\kappa^+}$. By the definition of $M$, we can conclude that $\varphi(z_0,\ldots,z_{n-1})$ holds in $M$.
\end{proof}

(i) First, let $\map{c}{[\kappa]^2}{2}$ be a function in $N$ that is $\Sigma_1(\kappa,z)$-definable in $N$ for some $z\in\HH{\kappa}^N$. Then $z\in M$ and the above claim implies that the same $\Sigma_1$-formula defines $c$ in $M$. Since $M$ is an elementary submodel of $\HH{\kappa^+}$ and $\kappa$ is weakly compact, $M$ contains a $c$-homogeneous subset of $\kappa$ that is unbounded in $\kappa$ and this subset is also an element of $N$.  
These computations show that in the structure $\langle j(\VV_\kappa),\in,A\rangle$, there is an inaccessible cardinal $\nu$ with the $\mathbf{\Sigma}_1$-colouring property with  $\VV_\nu\models\Psi(j(A)\cap\nu)$. With the help of a universal $\Sigma_1$-formula, this statement can be expressed by a first-order statement in $N$ that only uses the parameters $j(\kappa)$ and $j(A)$. By elementarity, there is an inaccessible cardinal $\mu<\kappa$ with the $\mathbf{\Sigma}_1$-colouring property such that $\VV_\mu\models\Psi(A\cap\mu)$ holds.

 (ii) Now, assume that $\kappa$ has the $\mathbf{\Sigma}_1$-club property and let $\map{c}{[\kappa]^{{<}\omega}}{\gamma}$ be a function in $N$ with $\gamma<\kappa$ that is  $\Sigma_1(\kappa,z)$-definable in $N$ for some $z\in\HH{\kappa}^N$. As above, we can conclude that $c$ is an element of $M$ and $\Sigma_1(\kappa,z)$-definable in that model. Since $M$ is an elementary submodel of $\HH{\kappa^+}$ and $\kappa$ has the $\Sigma_1$-club property, elementarity implies that $M$ contains a $c$-homogeneous set that is closed and unbounded in $\kappa$. As in (i), these computations show that in $\langle j(\VV_\kappa),\in,A\rangle$, there is an inaccessible cardinal $\nu$ with the $\mathbf{\Sigma}_1$-club property and the property that $\VV_\nu\models\Psi(j(A)\cap\nu)$ holds.  
\end{proof}

Next, we show that certain regular limits of cardinals with large cardinal properties stronger than weak compactness are also examples of inaccessible cardinals with the $\mathbf{\Sigma}_1$-club property. 
The following lemma also shows that successors of singular cardinals $\nu$ can possess the $\Sigma_1(z)$-club property for all parameters $z$ in $\HH{\nu}$.

\begin{lemma}\label{lemma:LimitMeasurablesClubProperty}
 Let $\nu$ be a strong limit cardinal, let $\kappa\in\{\nu,\nu^+\}$ be a regular cardinal and let $\delta<\nu$ be a measurable cardinal with $\cof{\nu}\neq\delta$. If $A$ is a bistationary subset of $\kappa$ and $z\in\HH{\delta}$, then the set $\{A\}$ is not $\Sigma_1(\kappa,z)$-definable. 
\end{lemma}

\begin{proof}
 Fix a normal ultrafilter $U$ on $\delta$ and let $$\langle\seq{M_\alpha}{\alpha\in\On} ,\seq{\map{j_{\alpha,\beta}}{M_\alpha}{M_\beta}}{\alpha\leq\beta\in\On}\rangle$$ denote the system of ultrapowers and elementary embeddings induced by $\langle\VV,\in,U\rangle$. Given $\alpha<\kappa$, set $\delta_\alpha=j_{0,\alpha}(\delta)$ and $\nu_\alpha=j_{0,\alpha}(\nu)$.

 \begin{claim*}
  If $\alpha<\kappa$, then $j_{0,\alpha}(\kappa)=\kappa$. 
 \end{claim*}

 \begin{proof}[Proof of the Claim]
  First, assume that $\kappa=\nu$. Since $\kappa$ is inaccessible in this case,  {\cite[Corollary 19.7.(c)]{MR1994835}} directly yields the statement of the claim. 

  Now, assume that $\kappa=\nu^+$. Note that, in order to prove the statement of the claim, it suffices to show that $\nu_\alpha<\kappa$ holds for all $\alpha<\kappa$, because we then have $\kappa\geq(\nu_\alpha^+)^{M_\alpha}=j_{0,\alpha}(\kappa)\geq\kappa$ for all $\alpha<\kappa$. If $\cof{\nu}>\delta$, then our assumptions imply that $\nu^\delta=\nu$ and therefore {\cite[Corollary 19.7.(a)]{MR1994835}} shows that $\nu_\alpha<(\nu^\delta\cdot\betrag{\alpha})^+=\kappa$ holds for all   $\alpha<\kappa$. In the other case, assume that $\cof{\nu}<\delta$. This assumption  implies that $\cof{\nu_\alpha}^{M_\alpha}=\cof{\nu}<\delta\leq\delta_\alpha$ for all $\alpha<\kappa$. We show $\nu_\alpha<\kappa$ by induction on $\alpha<\kappa$. Assume that $\alpha=\bar{\alpha}+1$. Then elementarity implies that $\nu_{\bar{\alpha}}$ is a strong limit cardinal greater than $\delta_\alpha$ in $M_{\bar{\alpha}}$. Therefore, {\cite[Corollary 19.7.(a)]{MR1994835}} shows that $j_{\bar{\alpha},\alpha}(\gamma)<\nu_{\bar{\alpha}}$ holds for all $\gamma<\nu_{\bar{\alpha}}$. By the above remarks, there is a function $\map{c}{\cof{\nu}}{\nu_{\bar{\alpha}}}$ in $M_{\bar{\alpha}}$ that is cofinal in $\nu_{\bar{\alpha}}$. But then elementarity implies that $\map{j_{\bar{\alpha},\alpha}(c)}{\cof{\nu}}{\nu_{\bar{\alpha}}}$ is cofinal in $\nu_{\alpha}$ and therefore $\nu_\alpha=\nu_{\bar{\alpha}}<\kappa$.  
Now, assume that $\alpha\in\Lim\cap\kappa$. Then our induction hypothesis implies that $\mu=\sup_{\bar{\alpha}<\alpha}\nu_{\bar{\alpha}}<\kappa$. If $\gamma<\nu_\alpha $, then there is an $\bar{\alpha}<\alpha$ and $\bar{\gamma}<\nu_{\bar{\alpha}}$ with $j_{\bar{\alpha},\alpha}(\bar{\gamma})=\gamma$. This shows that $\betrag{\nu_\alpha}\leq\mu\cdot\betrag{\alpha}<\kappa$. 
 \end{proof}

Fix a $\Sigma_1$-formula $\varphi(v_0,v_1,v_2)$ and $z\in\HH{\delta}$ with the property that there is a unique subset $A$ of $\kappa$ with the property that $\varphi(A,\kappa,z)$ holds.

\begin{claim*}
 If $\alpha<\kappa$, then $j_{0,\alpha}(A)=A$. 
\end{claim*}

\begin{proof}
 By the above claim and elementarity, we know that $\varphi(j_{0,\alpha}(A),\kappa,z)$ holds in $M_\alpha$. But then $\Sigma_1$-upwards absoluteness implies that $\varphi(j_{0,\alpha}(A),\kappa,z)$ holds in $\VV$ and the uniqueness of $A$ yields the statement of the claim. 
\end{proof}

Let $C$  be the club subset $\Set{\delta_\alpha}{\alpha<\kappa}$ of $\kappa$. If $\delta\in A$, then $\delta_\alpha=j_{0,\alpha}(\delta)\in j_{0,\alpha}(A)=A$ for all $\alpha<\kappa$ and hence $C\subseteq A$. In the other case, if $\delta\notin A$, then the same argument shows that $A\cap C=\emptyset$.  
\end{proof}

\begin{corollary}
 Regular limits of measurable cardinals have the $\mathbf{\Sigma}_1$-club property. \qed 
\end{corollary}

Next, we use $\omega_1$-iterable cardinals to provide more examples of non-weakly compact cardinals with the $\mathbf{\Sigma}_1$-club property. 
Again, the results of {\cite[Section 6]{MR3694344}} already provide this statement for $\Sigma_1(\kappa)$-definable subsets of $\kappa$.

\begin{theorem}
 Stationary limits of $\omega_1$-iterable cardinals have the $\mathbf{\Sigma}_1$-club property.  
\end{theorem}

\begin{proof}
 Pick such a cardinal $\kappa$, a $\Sigma_0$-formula $\varphi(v_0,\ldots,v_3)$ and $z\in\HH{\kappa}$ such that there is a unique $A\subseteq\kappa$ with the property that $\exists x ~ \varphi(A,x,\kappa,z)$ holds. By constructing a continuous elementary chain of elementary submodels of $\HH{\kappa^+}$ of cardinality less than $\kappa$, we  find an $\omega_1$-iterable cardinal $\delta<\kappa$ and an elementary substructure $X$ of $\HH{\kappa^+}$ of cardinality $\delta$ such that $\kappa,\tc{\{z\}},A\in X$, $\delta=\kappa\cap X$ and $z\in\HH{\delta}$.  Let $\map{\pi}{X}{M}$ denote the corresponding transitive collapse. 
Since $\delta$ is $\omega_1$-iterable, there is a transitive $\ZFC^-$-model $N$ and a weakly amenable $N$-ultrafilter $U$ on $\delta$ such that $\delta,M\in N$ and $\langle N,\in,U\rangle$ is $\omega_1$-iterable. Then $\tc{\{z\}},\pi(A)\in N$ and a combination of elementarity and $\Sigma_1$-upwards absoluteness implies that $N\models\exists x ~ \varphi(\pi(A),x,\delta,z)$. 
 Let $$\langle\seq{N_\alpha}{\alpha\in\On} ,\seq{\map{j_{\alpha,\beta}}{N_\alpha}{N_\beta}}{\alpha\leq\beta\in\On}\rangle$$ denote the  system of ultrapowers and elementary embeddings induced by $\langle N,\in,U\rangle$.  Then we have $j_{0,\kappa}(\delta)=\kappa$ and elementarity implies  $N_\kappa\models\exists x ~ \varphi((j_{0,\kappa}\circ\pi)(A),\kappa,z)$. But then $\varphi((j_{0,\kappa}\circ\pi)(A),\kappa,z)$ holds in $\VV$ and hence we get $A=(j_{0,\kappa}\circ\pi)(A)$. This shows that $A,\tc{\{z\}}\in N_\kappa\models\exists x ~ \varphi(A,x,\kappa,z)$ and, since there is a weakly amenable $N_\kappa$-ultrafilter $U_\kappa$ on $\kappa$ such $\langle N_\kappa,\in,U_\kappa\rangle$ is $\omega_1$-iterable, we can apply Lemma \ref{lemma:IterableModelsClubProperty} to conclude that $A$ either contains a club subset of $\kappa$ or is disjoint from such a subset.  
\end{proof}


\section{Successors of singular cardinals}\label{section:SuccSingular}

In this short section, we study the extend of definable partition properties at successors of singular cardinals. By combining Corollary \ref{corollary:2ImpliesN} with the following result  of Cummings, S. Friedman, Magidor, Rinot  and Sinapova from \cite{DefSuccSingCardinals}, it can be shown that the consistency of the existence of singular strong limit cardinal of countable cofinality whose successor has the $\mathbf{\Sigma}_2$-colouring property can be established from strong large cardinal assumptions.

 \begin{theorem}[\cite{DefSuccSingCardinals}]\label{theorem:HODSingular}
  Assume that $\nu$ is a singular limit of supercompact cardinals with $\cof{\nu}=\omega$ and $\kappa>\nu$ is supercompact.  Then there is a generic extension $\VV[G]$ of the ground model $\VV$ such that the following statements hold: 
  \begin{enumerate}[leftmargin=0.9cm]
   \item The models $\VV$ and $\VV[G]$ have the same bounded subsets of $\nu$. 

   \item Every infinite cardinal $\mu$ with $\mu\leq\nu$ or $\mu\geq\kappa$ is preserved in $\VV[G]$. 

   \item $\kappa=(\nu^+)^{\VV[G]}$. 

   \item If $z\in\POT{\nu}^{\VV[G]}$, then $\kappa$ is supercompact in $\HOD_z^{\VV[G]}$.
\end{enumerate}
 \end{theorem}

In contrast, results of Shelah show that that the $\mathbf{\Sigma}_2$-colouring property always fails at successors of singular strong limit cardinals of uncountable cofinality.

\begin{proposition}
 Let $\nu$ be a singular strong limit cardinal of uncountable cofinality. If $z\subseteq\nu$ with $\HH{\nu}\subseteq\LL[z]$, then no regular cardinal less than or equal to $2^\nu$ has the $\Sigma_2(\nu,z)$-colouring property.  
\end{proposition}

\begin{proof}
 A result of Shelah from \cite{MR1462202} (see also {\cite[Section 2]{DefSuccSingCardinals}}) shows that $\POT{\nu}\subseteq\HOD_z$. Pick a regular cardinal $\kappa\leq 2^\nu$. Then $\kappa\leq (2^\nu)^{\HOD_z}$ and therefore $\HOD_z$ contains a subset $A$ of $\kappa$ with $(2^\nu)^{\LL[A]}\geq\kappa$. By Proposition \ref{proposition:IntersectionsWithHOD}, there is such a subset $A$ of $\kappa$ with the property that the set $\{A\}$ is $\Sigma_2(\nu,z)$-definable.  But then Corollary \ref{corollary:InaccessibleInLA} implies that $\kappa$ does not have the $\Sigma_2(\nu,z)$-colouring property.  
\end{proof}

We close this section by showing that the validity of the $\Sigma_1$-colouring property at the successor of a singular cardinal has much larger consistency strength than the corresponding statement for successors of regular cardinals.

\begin{lemma}
 If there is a singular cardinal $\nu$ such that the cardinal $\nu^+$ has the $\Sigma_1$-colouring property, then there is an inner model with a measurable cardinal. 
\end{lemma}

\begin{proof}
 Assume that the above conclusion fails. Set $\kappa=\nu^+$ and let $\KK^{DJ}$ denote the \emph{Dodd--Jensen core model}. 
 
 \begin{claim*}
  There is a subset $A$ of $\kappa$ with the set property that the set $\{A\}$ is $\Sigma_1(\kappa)$-definable and there is an ordinal $\lambda<\kappa$ and a sequence $\seq{\map{s_\gamma}{\lambda}{\gamma}}{\gamma<\kappa}$ of surjection such that 
   \begin{equation}\label{equation:CodeSubsetSurjectionsSuccessor}
     A ~ = ~ \Set{\goedel{\gamma}{\alpha,s_\gamma(\alpha)}}{\gamma<\kappa, ~ \alpha<\lambda}.
    \end{equation}
 \end{claim*}
 
 \begin{proof}[Proof of the Claim]
  By our assumption,  the \emph{Covering Theorem} for $\KK$ (see {\cite[Theorem 5.17]{MR661475}}) implies that $\kappa=(\nu^+)^\KK$ and this shows that there is a sequence $\seq{\map{s_\gamma}{\nu}{\gamma}}{\gamma<\kappa}$ of surjections that is an element of $\KK$. 
 
 First, assume that there are no \emph{iterable premice} (see {\cite[Section 1]{MR730856}}). Then results of Dodd and Jensen (see  {\cite[p. 238]{MR730856}}) show that $\KK^{DJ}=\LL$. Let $A$ denote the $<_\LL$-least subset of $\kappa$ with the property that there is a $\lambda<\kappa$ and a sequence $\seq{\map{s_\gamma}{\lambda}{\gamma}}{\gamma<\kappa}$ of surjections with (\ref{equation:CodeSubsetSurjectionsSuccessor}). Then it is easy to see that the  set $\{A\}$ is $\Sigma_1(\kappa)$-definable. 
  
  Next, if $\KK\neq\LL$, then the results of Dodd and Jensen mentioned above show that $\KK^{DJ}$ is equal to the union of the \emph{lower parts $lp(M)$} of all iterable premice $M$. 
 Let $\calA$ denote the class of all subsets $A$ of $\kappa$ with the property that there is an iterable premouse $M=\LL_\alpha[F]$ such that $A\in lp(M)$ and $A$ is the $<_{\LL[F]}$-minimal subset of $\kappa$ in $M$ with the property that there is an ordinal $\lambda<\kappa$ and a sequence $\seq{\map{s_\gamma}{\lambda}{\gamma}}{\gamma<\kappa}$ in $lp(M)$ with (\ref{equation:CodeSubsetSurjectionsSuccessor}). Then our assumptions imply that $\calA$ is non-empty and, since the proof of {\cite[Lemma 2.3]{GoodWO}} shows that the class of all iterable premice is $\Sigma_1(\kappa)$-definable, we know that $\calA$ is definable in the same way. But then a comparison argument (see {\cite[Lemma 1.12.(7)]{MR730856}}) shows that $\calA$ consists of a single subset of $\kappa$.  
 \end{proof}
 
 Let $A$ be the subset of $\kappa$ given by the above claim. Then $\kappa$ is not a limit cardinal in $\LL[A]$ and therefore Corollary \ref{corollary:InaccessibleInLA} shows that $\kappa$ does not have the $\Sigma_1$-colouring property. 
\end{proof}


\section{Definable Homeomorphisms}\label{section:DefHomeo}

We present the results that were the initial motivation for the work presented in this paper. Remember that, given an uncountable regular cardinal $\kappa$, the \emph{generalized Baire space} of $\kappa$ consists of the set ${}^\kappa\kappa$ of all functions from $\kappa$ to $\kappa$ equipped with the topology whose basic open sets consist of all extensions of functions of the form $\map{s}{\alpha}{\kappa}$ with $\alpha<\kappa$. The \emph{generalized Cantor space} of $\kappa$ is the subspace of ${}^\kappa\kappa$ given by the set ${}^\kappa 2$ of all binary functions. 
 A classical result of Hung and Negrepontis from \cite{MR0367930} then shows that an uncountable regular cardinal $\kappa$ is weakly compact if and only if the generalized Baire space ${}^\kappa\kappa$ is not homeomorphic to the generalized Cantor space ${}^\kappa 2$ of $\kappa$ . 
Motivated by this characterization, Andretta and Motto Ros recently showed that the theory $\ZF+\DC+\AD$ proves that the generalized Baire space of $\omega_1$ is not homeomorphic to the generalized Cantor space of $\omega_1$ (see {\cite[Section 6.1]{2016arXiv160909292A}}). By combining this result with work of Woodin on the $\Pi_2$-maximality of the $\PPP_{max}$-extension of $\LL(\RRR)$ (see {\cite[Lemma 2.10 \& Theorem 7.3]{MR2768703}}), one can directly conclude that the existence of infinitely many Woodin cardinals with a measurable cardinal above them all implies that no homeomorphism between the generalized Baire space of $\omega_1$ and the generalized Cantor space of $\omega_1$ is definable by a $\Sigma_1$-formula that only uses the cardinal $\omega_1$ and elements of $\HH{\omega_1}$ as parameters, because Woodin's results show that the same formula defines a homeomorphism of these spaces in $\LL(\RRR)$. 
The question whether the above conclusion can be derived from weaker large cardinal assumptions was the initial motivation for the work presented in this paper. In combination with Theorem \ref{theorem:Omega1Sigma1wc}, the following lemma answers this question affirmatively.

\begin{lemma}
 If $\kappa$ is an uncountable regular cardinal with the $\Sigma_n(z)$-colouring property, then no homeomorphism between ${}^\kappa\kappa$ and ${}^\kappa 2$ is $\Sigma_n(\kappa,z)$-definable. 
\end{lemma}

\begin{proof}
 Assume that there is a $\Sigma_1$-formula $\varphi(v_0,\ldots,v_3)$ such that there is a homeomorphism $\map{h}{{}^\kappa\kappa}{{}^\kappa 2}$ with the property that for every $x\in{}^\kappa\kappa$, the function $h(x)$ is the unique set $y$ such that $\varphi(\kappa,x,y,z)$ holds. Given $\alpha<\kappa$, let $U_\alpha$ denote the open subset of ${}^\kappa\kappa$ consisting of all functions $x\in{}^\kappa\kappa$ with $x(0)=\alpha$ and let $x_\alpha$ denote the unique element of $U_\alpha$ with $x_\alpha(\beta)=0$ for all $0<\beta<\kappa$. If $\alpha<\kappa$, then $h[U_\alpha]$ is an open subset of ${}^\kappa 2$ that contains $h(x_\alpha)$ and hence there is a $\gamma<\kappa$ with the property that $h[U_\alpha]$ contains all extensions of $h(x_\alpha)\restriction\gamma$ in ${}^\kappa 2$. But this shows that for all $\alpha<\kappa$, there is a unique minimal $\gamma_\alpha<\kappa$ with the property that $$\alpha=\beta ~ \Longleftrightarrow ~ {h(x_\alpha)\restriction\gamma_\alpha}\subseteq h(x_\beta)$$ holds for all $\beta<\kappa$. Then the resulting map $$\Map{\iota}{\kappa}{{}^{{<}\kappa}2}{\alpha}{h(x_\alpha)\restriction\gamma_\alpha}$$ is an injection and our assumptions imply that it is $\Sigma_n(\kappa,z)$-definable. Since all elements of $\ran{\iota}$ are pairwise incompatible, we can use Lemma \ref{lemma:CharacterizationsPartitionProperty} to conclude that $\kappa$ does not have the $\Sigma_n(z)$-colouring property.  
\end{proof}

In combination with Theorem \ref{theorem:Omega1Sigma1wc}, the above lemma shows that the existence of a measurable cardinal above a Woodin cardinal implies that no homeomorphism between the generalized Baire space of $\omega_1$ and the generalized Cantor space of $\omega_1$ is definable by a $\Sigma_1$-formula with parameters in $\HH{\omega_1}\cup\{\omega_1\}$. 
The next lemma shows that the implication  proven  above can be reversed in certain models of set theory.

\begin{lemma}\label{lemma:DefHomeo}
 Let $\kappa$ be an uncountable regular cardinal with the property that there is a good $\Sigma_n(\kappa,y)$-well-ordering of $\HH{\kappa}$ of length $\kappa$. If $\kappa$ does not have the $\Sigma_n(z)$-colouring property, then there is a $\Sigma_n(\kappa,y,z)$-definable homeomorphism between ${}^\kappa\kappa$ and ${}^\kappa 2$. 
\end{lemma}

\begin{proof}
 By Lemma \ref{lemma:CharacterizationsPartitionProperty}, our assumptions imply the existence of a $\Sigma_n(\kappa,z)$-definable injection $\map{\iota}{\kappa}{{}^{{<}\kappa}2}$ with the property that for all $x\in{}^\kappa 2$, there is an $\alpha<\kappa$ such that there is no $\beta<\kappa$ with $x\restriction\alpha\subseteq\iota(\beta)$. 
  Set $T=\Set{t\in{}^{{<}\kappa}2}{\exists\alpha<\kappa ~ t\subseteq\iota(\alpha)}$ and define $\partial T$ to be the set of all $t\in{}^{{<}\kappa}2\setminus T$ with the property that $t\restriction\alpha\in T$ holds for all $\alpha<\length{t}$. Then our assumptions imply that for every $x\in{}^\kappa 2$, there is a unique $t_x\in\partial T$ with $t_x\subseteq x$. 
 
 \begin{claim*}
  The set $\partial T$ has cardinality $\kappa$. 
 \end{claim*}
 
 \begin{proof}[Proof of the Claim]
  Assume that $\partial T$ has cardinality less than $\kappa$. Since $\iota$ is an injection, there is an $\alpha<\kappa$ such that there is no $t\in\partial T$ with $\iota(\alpha)\subseteq\partial T$. Pick $x\in{}^\kappa 2$ with $\iota(\alpha)\subseteq x$. Then $t_x\nsubseteq\iota(\alpha)$ and therefore $\iota(\alpha)\subseteq t_x\in\partial T$, a contradiction.  
 \end{proof}

Note that our assumption on $\iota$ imply that the set $\partial T$ is $\Sigma_n(\kappa,z)$-definable. By the above claim, the existence of a good $\Sigma_n(\kappa,y)$-well-ordering of $\HH{\kappa}$ of length $\kappa$ then yields the existence of a $\Sigma_n(\kappa,y,z)$-definable bijection $\map{b}{\kappa}{\partial T}$. Given $y\in{}^\kappa\kappa$, we can then find a unique element $h(y)$ of ${}^\kappa 2$ with the property that there is a continuous increasing sequence $\seq{\beta_\alpha}{\alpha<\kappa}$ of ordinals less than $\kappa$ with $\beta_0=0$, $\beta_{\alpha+1}=\beta_\alpha+\length{b(y(\alpha))}$ and $h(y)(\beta_\alpha+\beta)=b(y(\alpha))(\beta)$ for all $\alpha<\kappa$ and $\beta<\length{b(y(\alpha))}$.

\begin{claim*}
 The map $\map{h}{{}^\kappa\kappa}{{}^\kappa 2}$ is a homeomorphism. 
\end{claim*}

\begin{proof}[Proof of the Claim]
 Given $x\in{}^\kappa 2$, there is a unique element $g(x)$ of ${}^\kappa\kappa$ with the property that there exists a sequence $\seq{x_\alpha}{\alpha<\kappa}$ of elements of ${}^\kappa 2$ and a continuous increasing sequence $\seq{\beta_\alpha}{\alpha<\kappa}$ of ordinals less than $\kappa$ such that the following statements hold: 
 \begin{enumerate}
  \item $x_0=x$ and $\beta_0=0$. 

  \item $\beta_{\alpha+1}=\beta_\alpha+\length{t_{x_\alpha}}$ and $b(g(x)(\alpha))=t_{x_\alpha}$ for all $\alpha<\kappa$. 

  \item $x_\alpha(\beta)=x(\beta_\alpha+\beta)$ for all $\alpha,\beta<\kappa$.  
 \end{enumerate}
 Then it is easy to check that $\map{g}{{}^\kappa 2}{{}^\kappa\kappa}$ and   $\map{h}{{}^\kappa\kappa}{{}^\kappa 2}$ are continuous functions with $g\circ h=\id_{{}^\kappa\kappa}$ and $h\circ g=\id_{{}^\kappa 2}$. 
\end{proof}

 Finally, the above construction ensure that $h$ is $\Sigma_n(\kappa,y,z)$-definable. 
\end{proof}

The above results show that the assumption $\VV=\LL$ or, more generally, $\VV=\KK^{DJ}$ implies that an uncountable regular cardinal $\kappa$ has the $\Sigma_1(z)$-colouring property if and only if there is no $\Sigma_1(\kappa,z)$-definable  homeomorphism between ${}^\kappa\kappa$ and ${}^\kappa 2$. 
Moreover, a combination of the above lemma with {\cite[Theorem 2]{MR2231126}} (as in the proof of Proposition \ref{proposition:BPFAfailureBoldface}) shows that $\BPFA$ implies that for every $z\subseteq\omega_1$ with $\omega_1=\omega_1^{\LL[z]}$, there is a $\Sigma_1(\omega_2,z)$-definable  homeomorphism between the generalized Baire space of $\omega_2$ and the generalized Cantor space of $\omega_2$. 
Finally, by combining the construction from the proof of Lemma \ref{lemma:DefHomeo} with arguments from the proof of {\cite[Theorem 5.2]{MR3694344}}, it is possible to show that the existence of a single Woodin cardinal is compatible with the existence of a $\Sigma_1(\omega_1)$-definable  homeomorphism between the generalized Baire space of $\omega_1$ and the generalized Cantor space of $\omega_1$.


\section{Open questions}\label{section:questions}

We close this paper by presenting some questions left open by the above results. The results of Section \ref{section:Omega2} show that both $\PFA^{++}$ and $\MM^{{+}\omega}$ do not imply that $\omega_2$ has the $\Sigma_1$-colouring property. Since these results rely on the fact that these forcing axioms are compatible with the existence of a well-ordering of the reals that is definable over $\langle\HH{\omega_2},\in\rangle$ and it is commonly expected that stronger forcing axioms imply the non-existence of such a well-ordering, it is natural to conjecture that such axioms also rule out the existence of simply definable partitions without large homogeneous sets.

\begin{question}
 Do very strong forcing axioms, like $\MM^{++}$ or $\MM^{+++}$ (defined by Viale in \cite{MR3486170}), imply that $\omega_2$ has the $\Sigma_1$-colouring property?
\end{question}

In contrast, the above results also leave open the possibility that \emph{Martin's Maximum} is not only compatible with a failure of the $\Sigma_1$-colouring property at $\omega_2$, but outright implies such a failure.

\begin{question}
 Is $\MM$ consistent with the statement that $\omega_2$ has the $\Sigma_1$-colouring property?  
\end{question}

While the results of Section \ref{section:Limits} provide many examples of inaccessible non-weakly compact cardinals with the $\Sigma_1$-colouring property, they leave open the question question whether  small inaccessible cardinals can possess this property.

\begin{question}
 Is it consistent that the first inaccessible cardinal has the $\mathbf{\Sigma}_1$-colouring property? 
\end{question}

Somewhat surprisingly, the above results show that successors of singular strong limit cardinals of uncountable cofinality never have the $\mathbf{\Sigma}_2$-colouring property. This leaves open the following question.

\begin{question}
 Is it consistent that the successor of a singular cardinal of uncountable cofinality has the $\mathbf{\Sigma}_1$-colouring property?
\end{question}


 \bibliographystyle{amsplain}
\bibliography{references}

\providecommand{\bysame}{\leavevmode\hbox to3em{\hrulefill}\thinspace}
\providecommand{\MR}{\relax\ifhmode\unskip\space\fi MR }
\providecommand{\MRhref}[2]{%
  \href{http://www.ams.org/mathscinet-getitem?mr=#1}{#2}
}
\providecommand{\href}[2]{#2}
\begin{thebibliography}{10}

\bibitem{MR716625}
Uri Abraham and Saharon Shelah, \emph{Forcing closed unbounded sets}, Journal
  of Symbolic Logic \textbf{48} (1983), no.~3, 643--657. \MR{716625}

\bibitem{2016arXiv160909292A}
Alessandro {Andretta} and Luca {Motto Ros}, \emph{{Classifying uncountable
  structures up to bi-embeddability}}, preprint.

\bibitem{MR2320944}
David Asper\'o, \emph{Guessing and non-guessing of canonical functions}, Ann.
  Pure Appl. Logic \textbf{146} (2007), no.~2-3, 150--179. \MR{2320944}

\bibitem{MR2310340}
Joan Bagaria and Roger Bosch, \emph{Generic absoluteness under projective
  forcing}, Fund. Math. \textbf{194} (2007), no.~2, 95--120. \MR{2310340}

\bibitem{MR776640}
James~E. Baumgartner, \emph{Applications of the proper forcing axiom}, Handbook
  of set-theoretic topology, North-Holland, Amsterdam, 1984, pp.~913--959.
  \MR{776640}

\bibitem{MR2267146}
Roger Bosch, \emph{Small definably-large cardinals}, Set theory, Trends Math.,
  Birkh\"auser, Basel, 2006, pp.~55--82. \MR{2267146}

\bibitem{MR2231126}
Andr{\'e}s~Eduardo Caicedo and Boban Veli{\v{c}}kovi{\'c}, \emph{The bounded
  proper forcing axiom and well orderings of the reals}, Math. Res. Lett.
  \textbf{13} (2006), no.~2-3, 393--408. \MR{2231126 (2007d:03076)}

\bibitem{MR2768691}
James Cummings, \emph{Iterated forcing and elementary embeddings}, Handbook of
  set theory. {V}ols. 1, 2, 3, Springer, Dordrecht, 2010, pp.~775--883.
  \MR{2768691}

\bibitem{DefSuccSingCardinals}
James Cummings, Sy-David Friedman, Menachem Magidor, Assaf Rinot, and Dima
  Sinapova, \emph{Ordinal definable subsets of singular cardinals}, to appear
  in the \emph{Israel Journal of Mathematics}.

\bibitem{MR661475}
Tony Dodd and Ronald~B. Jensen, \emph{The covering lemma for {$K$}}, Ann. Math.
  Logic \textbf{22} (1982), no.~1, 1--30. \MR{661475}

\bibitem{MR645907}
Hans-Dieter Donder, Ronald~B. Jensen, and Bernd~J. Koppelberg, \emph{Some
  applications of the core model}, Set theory and model theory ({B}onn, 1979),
  Lecture Notes in Math., vol. 872, Springer, Berlin-New York, 1981,
  pp.~55--97. \MR{645907}

\bibitem{MR730856}
Hans-Dieter Donder and Peter Koepke, \emph{On the consistency strength of
  ``accessible''\ {J}\'onsson cardinals and of the weak {C}hang conjecture},
  Ann. Pure Appl. Logic \textbf{25} (1983), no.~3, 233--261. \MR{730856}

\bibitem{MR0141603}
Paul {Erd\H{o}s} and Andras Hajnal, \emph{Some remarks concerning our paper
  ``{O}n the structure of set-mappings''. {N}on-existence of a two-valued
  {$\sigma$}-measure for the first uncountable inaccessible cardinal}, Acta
  Math. Acad. Sci. Hungar. \textbf{13} (1962), 223--226. \MR{0141603}

\bibitem{MR2830415}
Victoria Gitman, \emph{Ramsey-like cardinals}, J. Symbolic Logic \textbf{76}
  (2011), no.~2, 519--540. \MR{2830415 (2012e:03110)}

\bibitem{MR2830435}
Victoria Gitman and Philip~D. Welch, \emph{Ramsey-like cardinals {II}}, J.
  Symbolic Logic \textbf{76} (2011), no.~2, 541--560. \MR{2830435
  (2012e:03111)}

\bibitem{MR0373889}
Serge Grigorieff, \emph{Intermediate submodels and generic extensions in set
  theory}, Ann. Math. (2) \textbf{101} (1975), 447--490. \MR{0373889}

\bibitem{MR1133077}
Kai Hauser, \emph{Indescribable cardinals and elementary embeddings}, J.
  Symbolic Logic \textbf{56} (1991), no.~2, 439--457. \MR{1133077}

\bibitem{MR0367930}
H.~H. Hung and S.~Negrepontis, \emph{Spaces homeomorphic to {$(2^{\alpha
  })_{\alpha }$}}, Bull. Amer. Math. Soc. \textbf{79} (1973), 143--146.
  \MR{0367930 (51 \#4172)}

\bibitem{MR0244036}
Thomas Jech, \emph{{$\omega _{1}$} can be measurable}, Israel J. Math.
  \textbf{6} (1968), 363--367 (1969). \MR{0244036}

\bibitem{MR1940513}
\bysame, \emph{Set theory}, Springer Monographs in Mathematics,
  Springer-Verlag, Berlin, 2003, The third millennium edition, revised and
  expanded. \MR{1940513}

\bibitem{MR560220}
Thomas Jech, Menachem Magidor, William~J. Mitchell, and Karel Prikry,
  \emph{Precipitous ideals}, J. Symbolic Logic \textbf{45} (1980), no.~1, 1--8.
  \MR{560220}

\bibitem{MR1994835}
Akihiro Kanamori, \emph{The higher infinite}, second ed., Springer Monographs
  in Mathematics, Springer-Verlag, Berlin, 2003, Large cardinals in set theory
  from their beginnings. \MR{1994835}

\bibitem{MR597342}
Kenneth Kunen, \emph{Set theory}, Studies in Logic and the Foundations of
  Mathematics, vol. 102, North-Holland Publishing Co., Amsterdam-New York,
  1980, An introduction to independence proofs. \MR{597342}

\bibitem{MR2069032}
Paul~B. Larson, \emph{The stationary tower}, University Lecture Series,
  vol.~32, American Mathematical Society, Providence, RI, 2004, Notes on a
  course by W. Hugh Woodin. \MR{2069032}

\bibitem{MR2474445}
\bysame, \emph{Martin's maximum and definability in {$H(\aleph_2)$}}, Ann. Pure
  Appl. Logic \textbf{156} (2008), no.~1, 110--122. \MR{2474445}

\bibitem{MR2768703}
\bysame, \emph{Forcing over models of determinacy}, Handbook of set theory.
  {V}ols. 1, 2, 3, Springer, Dordrecht, 2010, pp.~2121--2177. \MR{2768703}

\bibitem{MR1791372}
Amir Leshem, \emph{On the consistency of the definable tree property on
  {$\aleph_1$}}, J. Symbolic Logic \textbf{65} (2000), no.~3, 1204--1214.
  \MR{1791372}

\bibitem{MR2987148}
Philipp L\"ucke, \emph{{$\Sigma^1_1$}-definability at uncountable regular
  cardinals}, J. Symbolic Logic \textbf{77} (2012), no.~3, 1011--1046.
  \MR{2987148}

\bibitem{MR3694344}
Philipp L\"ucke, Ralf Schindler, and Philipp Schlicht,
  \emph{{$\Sigma_1(\kappa)$}-definable subsets of {${\rm H}(\kappa^+)$}}, J.
  Symb. Log. \textbf{82} (2017), no.~3, 1106--1131. \MR{3694344}

\bibitem{GoodWO}
Philipp L\"ucke and Philipp Schlicht, \emph{Measurable cardinals and good
  {$\Sigma_1(\kappa)$}-wellorderings}, to appear in {\emph{Mathematical Logic
  Quarterly}}.

\bibitem{MR1300637}
William~J. Mitchell and John~R. Steel, \emph{Fine structure and iteration
  trees}, Lecture Notes in Logic, vol.~3, Springer-Verlag, Berlin, 1994.
  \MR{1300637 (95m:03099)}

\bibitem{MR2768699}
Ernest Schimmerling, \emph{A core model toolbox and guide}, Handbook of set
  theory. {V}ols. 1, 2, 3, Springer, Dordrecht, 2010, pp.~1685--1751.
  \MR{2768699}

\bibitem{MR2817562}
Ian Sharpe and Philip~D. Welch, \emph{Greatly {E}rd{\H o}s cardinals with some
  generalizations to the {C}hang and {R}amsey properties}, Ann. Pure Appl.
  Logic \textbf{162} (2011), no.~11, 863--902. \MR{2817562}

\bibitem{MR1462202}
Saharon Shelah, \emph{Set theory without choice: not everything on cofinality
  is possible}, Arch. Math. Logic \textbf{36} (1997), no.~2, 81--125.
  \MR{1462202}

\bibitem{MR1257469}
John~R. Steel, \emph{Inner models with many {W}oodin cardinals}, Ann. Pure
  Appl. Logic \textbf{65} (1993), no.~2, 185--209. \MR{1257469 (95c:03132)}

\bibitem{MR2768698}
\bysame, \emph{An outline of inner model theory}, Handbook of set theory.
  {V}ols. 1, 2, 3, Springer, Dordrecht, 2010, pp.~1595--1684. \MR{2768698}

\bibitem{MR2355670}
Stevo Todor\v{c}evi\'{c}, \emph{Walks on ordinals and their characteristics},
  Progress in Mathematics, vol. 263, Birkh\"auser Verlag, Basel, 2007.
  \MR{2355670}

\bibitem{MR3486170}
Matteo Viale, \emph{Category forcings, {$\mathrm{MM}^{+++}$}, and generic
  absoluteness for the theory of strong forcing axioms}, J. Amer. Math. Soc.
  \textbf{29} (2016), no.~3, 675--728. \MR{3486170}

\end{thebibliography}


\end{document}